\documentclass{article}
\usepackage{geometry}
\usepackage[T1]{fontenc}
\usepackage[dvipsnames]{xcolor}
\usepackage{mathtools}
\usepackage{amsmath}
\usepackage{amsthm}
\usepackage{amssymb}
\usepackage{setspace}
\usepackage[colorlinks, linkcolor = Green!80!black]{hyperref}
\makeatletter

\numberwithin{equation}{section}
\numberwithin{figure}{section}

\theoremstyle{plain}
\newtheorem{thm}{Theorem}[section]
\newtheorem{prop}[thm]{Proposition}
\newtheorem{lem}[thm]{Lemma}
\newtheorem{cor}[thm]{Corollary}
\newtheorem{fact}[thm]{Fact}

\theoremstyle{definition}
\newtheorem{defn}[thm]{Definition}
\newtheorem*{defn*}{Definition}

\makeatother

\DeclareMathOperator{\E}{\mathbb{E}}
\DeclareMathOperator{\Var}{\mathrm{Var}}
\DeclareMathOperator{\supp}{\mathrm{supp}}
\newcommand{\Ind}[1]{\mathbf{1}\{#1\}}
\newcommand{\set}[1]{\{#1\}}
\newcommand{\deq}{\stackrel{d.}{=}}
\newcommand{\eRR}{\mathbb{\overline{R}}}
\newcommand{\pois}{\mathrm{Pois}}

\newcommand{\MM}{\mathcal{M}}
\renewcommand{\SS}{\mathsf{S}}
\newcommand{\vX}{\vec{X}}

\newcommand{\iA}{\mathsf{A}}
\newcommand{\ia}{\mathsf{a}}
\newcommand{\iB}{\mathsf{B}}
\newcommand{\ib}{\mathsf{b}}
\newcommand{\iP}{\mathsf{P}}
\newcommand{\iU}{\mathsf{U}}

\newcommand{\itA}{\widetilde{\iA}}
\newcommand{\itB}{\widetilde{\iB}}
\newcommand{\itP}{\widetilde{\iP}}
\newcommand{\itb}{\tilde{\ib}}
\newcommand{\iD}{\mathsf{D}}
\newcommand{\tpois}{\mathrm{tPois}}

\newcommand{\tq}{\tilde{q}}
\renewcommand{\phi}{\varphi}

\begin{document}

\allowdisplaybreaks
\raggedbottom

\title{Reconstruction of colourings without freezing}
\author{Allan Sly\thanks{Princeton University and University of California, Berkeley. Supported by NSF grants DMS-1208338 and DMS-1352013. Email:sly@stat.berkeley.edu} \and Yumeng Zhang\thanks{University of California, Berkeley. Email:ymzhang@berkeley.edu} }

\maketitle
\begin{abstract}
We prove that reconstruction in the $k$-colouring model occurs strictly below the threshold for freezing for large $k$.
\end{abstract}

\section{Introduction}

The study of broadcast models or spin systems on trees arises naturally in many areas including probability, combinatorics and statistical physics as well as  in more applied settings such as computational evolutionary biology and information theory.  The so-called reconstruction problem  asks when the mutual information between the root and the spins at level $\ell$ is bounded away from 0 as $\ell \to \infty$ and thus can be viewed as a type of point to set dependence  (see definition in \S~\ref{sec:Defin}).  It emerges in numerous settings, for example in biology it determines a phase transition for the information requirements for phylogenetic reconstruction~\cite{daskalakis2011evolutionary}.

Here we are most interested in the role the  reconstruction threshold plays in the study of random constraints satisfaction problems (rCSPs).  It has been shown that in a range of rCSPs such as random colourings of random graphs, the space of solutions undergoes what physicists call a dynamical phase transition in which the space of solutions splits into exponentially many small, isolated clusters~\cite{achlioptas2008algorithmic}.  This transition also seems closely related to computational barriers for algorithms for finding solutions.  It has been conjectured that the threshold for this transition is exactly the reconstruction threshold and this is known up to first order asymptotic.

Locating the exact reconstruction threshold has only been achieved in a small number of spin systems, the symmetric~\cite{evans2000broadcasting} and near-symmetric binary channels~\cite{borgs2006kesten} and the three state symmetric channel with large degrees~\cite{sly2011reconstruction}.  For the $k$-colouring model only bounds are known which match in the first and second order asymptotic term.  On the $d$-regular tree the model is non-reconstructible whenever~\cite{bhatnagar2011reconstruction,sly2009reconstruction},
\begin{equation}
d\le k(\log k+\log\log k+1-\log2+o_{k}(1)).
\end{equation}
The best previous bound for reconstruction is when
\begin{equation}\label{eq:freezing}
d\ge k(\log k+\log\log k+1+o_{k}(1))
\end{equation}
by~\cite{mossel2003information,semerjian2008freezing}.  This uses the following simple algorithm; it reconstructs the root only when it is uniquely determined by the leaves, in which case we say the root is \emph{frozen}.  This can be implemented and analysed using a simple recursion  and leaves a gap of width just $k\log2$. It is known that~(\ref{eq:freezing}) is tight for freezing of the root so one natural question to ask is whether  reconstruction is possible when the root is not frozen. We  answer this in the positive showing that the $k$-colouring model is still reconstructible for parameters in a small but non-vanishing region of width $\delta k$ below the freezing threshold.

Interpreted in the setting of random colourings on random graphs this is opens a number of tantalising questions.  It suggests a range of parameters in which there is clustering of colourings but where the clusters are unfrozen meaning that all vertices can take every possible colour within the cluster.  It remains an important question to understand what leads to the computational difficulty in colouring random graphs, the onset of clustering or of freezing.  Our result separates these two transitions making this distinction of keen importance.

\subsection{Definition and Main Results}\label{sec:Defin}

The broadcast model on trees is the process where information is sent from the roots downward, along edges acting as noisy channels, to the leaves of the trees. Given a tree $T=(V,E)$, a finite set $[k]=\{1,\dots,k\}$ of $k$ values
and a $[k]\times[k]$ probability matrix $M$ as the noisy channel, the broadcast model on tree $T$ is the probability measure on the space of configurations $[k]^{V}$ defined as follows: The spin $\sigma_{\rho}$
at the root $\rho$ is chosen according to the stationary distribution of $M$, denoted by $\pi$. Then for each vertex $v\in T$ with parent $u$, the spin $\sigma_{v}$ is chosen according to the conditional distribution $P(\sigma_{v}=i\mid\sigma_{u}=j)=M(i,j)$. In this paper we will focus on the colouring model with alphabet $[k]$ and probability matrix $M(i,j)=\frac{1}{k-1}1\{i\neq j\}$.

Equivalently, one can also define the colouring model by its Gibbs measure. A proper $k$-colouring of the graph $G=(V,E)$ is a configuration
$\sigma:V\to[k]$ such that for every edge $e=(u,v)\in E$, $\sigma_{u}\neq\sigma_{v}$.
The (free) Gibbs measure of random colourings is given by the uniform
measure
\[
P(\sigma)=\frac{1}{Z}\prod_{e=(u,v)\in E}1\{\sigma_{u}\neq\sigma_{v}\},
\]
where $Z$ is the normalizing constant equaling to the number of proper colourings of $G$.

For technical convenience and also of independent interest, we allow randomness in the underlying trees. For any probability distributions $\xi$ on the set of non-negative integers $\mathbb{Z}_{+}$, we let $\mathcal{T}_{\xi}$ denote the distribution of Galton-Watson tree with offspring distribution
$\xi$. Two special cases of interest are the $d$-ary tree $\mathcal{T}_{d}$
and the Galton-Watson tree $\mathcal{T}_{\pois(d)}$ with Poisson
offspring distribution of average degree $d$. They are the natural tree models to study with regard to random $d$-regular graphs and Erd\H{o}s-R\'{e}nyi
random graphs respectively. The definition of broadcast model can be easily generalized to the (first finite levels of) Galton-Watson trees.

Given a (possibly random) infinite tree, the reconstruction problem asks if the distribution of the state of the root is affected by the configuration on the $n$'th level as $n$ goes to infinity. More precisely, let $T_{n}$ be the first $n$ levels of tree $T$ and $L_{n}$ be its set of vertices at level $n$.
Write $T_n = T, L_{n}=\varnothing$ if $T$ has fewer than $n$ levels.
\begin{defn*}[Reconstruction]
    Given a family of Galton-Watson trees $\mathcal{T}_{\xi}$, we say that the $k$-colouring model is reconstructible for $\mathcal{T}_{\xi}$
if there exist $i,j\in[k]$ such that,
\[
    \limsup_{n\to\infty}
    \E_{T\sim\mathcal{T}_{\xi}}
    d_{\mathrm{TV}} (
        P(\sigma_{L_{n}}=\cdot\mid T,\sigma_{\rho}=i),
        P(\sigma_{L_{n}}=\cdot\mid T,\sigma_{\rho}=j)
    )
    >0,
\]
where $d_{\mathrm{TV}}$ is the total-variation distance. Otherwise we say that the model is non-reconstructible.
\end{defn*}
Non-reconstruction implies that on average the configurations on the distant levels have a vanishing effect on the root. Equivalently, it corresponds
to the mutual information between the root and the leaves going to 0 (see e.g.\ \cite{mossel2004survey} for more equivalent definitions). The freezing threshold is defined as follows:
\begin{defn*}[Freezing]
    Given a family of Galton-Watson tree $\mathcal{T}_{\xi}$, we say that the $k$-colouring model is frozen for $\mathcal{T}_{\xi}$ if
\[
\limsup_{n\to\infty}P_{T\sim\mathcal{T}_{\xi}}(\sigma_\rho\textup{ is uniquely determined by }\sigma_{L_{n}})>0.
\]

\end{defn*}
The exact location of freezing threshold for Poisson tree $\mathcal{T}_{\pois(d)}$
has been calculated in \cite{molloy2012freezing}. Following a similar
calculation for $\mathcal{T}_{d}$, one can show that for $k\ge k_{0}$, the $k$-colouring model is frozen if and only if
\[
d>d_{k}^{f}:=\begin{cases}
\inf_{x>0}x\log^{-1}\left(1-\frac{(1-e^{-x})^{k}}{k-1}\right) & \mathcal{T}_{d}\\
\inf_{x>0}\frac{(k-1)x}{(1-e^{-x})^{k}} & \mathcal{T}_{\mathrm{Pois}(d)}
\end{cases}=k(\log k+\log\log k+1+o_{k}(1)).
\]
It is easy to see that the $k$-colouring problem is reconstructible on $\mathcal{T}_{\xi}$ if it is frozen. Indeed, the freezing
threshold gives the best known upper bound for reconstruction threshold
with the only exception of $d=5$ and $k=14$, in which case reconstruction
is proved in \cite{mezard2006reconstruction} using a variational
principle. The main result of this paper is the following theorem which implies that the reverse statement is not true. Throughout we will assume that $k$ exceeds a large enough absolute constant $k_0$, where the exact value of $k_0$ may vary from place to place.
\begin{thm} \label{thm:col_main_rec}
There exists a constant $\beta^{*}<1$ such that for any $k\ge k_0$ the $k$-colouring model is reconstructible for both $\mathcal{T}_{d}$
and $\mathcal{T}_{\pois(d)}$ for $d$ satisfying
\begin{equation}
d\ge k(\log k+\log\log k+\beta^{*}).\label{eq:col_main_recregime}
\end{equation}

\end{thm}
For a complete picture, it has been shown \cite{bhatnagar2011reconstruction,sly2009reconstruction}
that the $k$-colouring problem is non-reconstructible on $d$-ary tree with
\[
d<k(\log k+\log\log k+1-\log2+o_{k}(1)),
\]
and the similar result extends to general Galton-Watson
trees under mild restrictions \cite{efthymiou2014reconstruction}. While numerical results
of \cite{zdeborova2007phase} suggest that the actual reconstruction
threshold has a constant term roughly in the middle of $1-\log2$
and $1$, for technical reasons we only show reconstruction for $\beta_*$ close to the freezing threshold $1$. Nonetheless, we believe that our result
is of interest because it suggests a distinct phase transition in the solution space evolution of rCSPs, the existence of which was previously unclear. We will address this point in detail in the next section.

\subsection{Motivation from Statistical Physics}

Random instances of constraint satisfaction problems (rCSPs) have been studied in different areas including theoretical computer science, probability theory, combinatorics and statistical physics. Much of our understanding of the problem over the last two decades comes from the replica/cavity method originally developed by statistical physicist in study of spin glasses, among which perhaps the two most important questions are\emph{ when does a rCSP have solution} and \emph{how can we find/sample one}.
Significant progresses have been made in the last couple of years towards the first question. Exact satisfiability thresholds have been established for $k$-NAESAT \cite{ding2014satisfiability}, maximum independent set \cite{ding2013maximum} and $k$-SAT \cite{ding2014proof}, and the $k$-colourability threshold has been located within an interval of length $(2\log2-1)$ \cite{coja2013upper,coja2013chasing}.

Meanwhile, on the algorithmic side, it has been observed for many models of interest that all polynomial-time algorithms fail to find solutions at densities far below the satisfiability threshold.
This algorithmic barrier is believed to be closely related to the phase transitions in the geometry of the set of solutions. Here we briefly review the heuristic phase diagram developed by statistical physicists \cite{krzakala2007gibbs,zdeborova2007phase}, as we
fix $k$ and increase the average connectivity $d$. The set of solutions start out as a well-connected component containing all but exponentially small fraction of solutions.
At the \emph{clustering threshold} $d_{\mathrm{clust}}$, the solution space splits into an exponential number of ``clusters" where clusters are well-connected inside but well-separated from each other, and no single cluster contains more than an exponentially small fraction of all solutions.
Then at the possibly higher value of $d$, namely the \emph{rigidity threshold} $d_\textup{rigid}$, typical clusters become ``frozen", i.e.\ a linear fraction of variables take the same value in all solutions of that cluster.
Finally at much larger values of $d$ come the condensation threshold and satisfiability threshold, which we will not go into details here.

These predictions have been partially verified in many cases. Apart from the results on satisfiability threshold mentioned before, Molloy~\cite{molloy2012freezing} proved that the rigidity phase transition coincide with the freezing threshold on trees, in the case of $k$-colouring. And in the prominent paper \cite{achlioptas2008algorithmic}, the authors proved that the solution space does split into exponentially many frozen clusters for $k$-colourings models and constraint densities $(1+o_{k}(1))k\log k\le d\le(2-o_{k}(1))k\log k$.

Among the different phase transitions mentioned above, it has been conjectured that the clustering threshold and the rigidity threshold are the two factors resulting the onset of hard random-CSP instances.  However different opinions exist on which one is more responsible \cite{mezard2002random, zdeborova2007phase,zdeborova2008constraint}, if any of them \cite{braunstein2016LDP}. And much is unclear about how they affect the performance of algorithms directly.
One difficulty lies in the fact that the two thresholds are extremely close to each other. According to the physics prediction, \cite{zdeborova2007phase}, both thresholds happen at $k(\log k+\log \log k + \alpha+ o(1))$ for different values of $\alpha$ and no evidence shows even at a heuristic level that such gap is indeed non-vanishing.
In fact, it has been widely believed that the clustering phase transition, marking the onset of long range correlation, coincides with the reconstruction threshold on trees \cite{montanari2011reconstruction}. If that is the case, then previous results in the reconstruction problems \cite{sly2009reconstruction} imply that the gap between the two thresholds can at most be $k(\log 2 +o(1))$ (compared to the leading term of $k\log k$).

We hope that the result of this paper can contribute to the understanding of colourings on random graphs in two directions. First, we show for the first time that the gap between reconstruction threshold and freezing threshold on trees is linear in $k$. This combined with the conjecture that reconstruction coincides with clustering strongly suggests a distinct phase where the solution space are clustered but non-frozen. It will be of great interest to analyze algorithms in this region. Secondly, the distributional recursion involved in the reconstruction problem  (known as the averaged 1RSB equation in physics jargon \cite{mezard2009information}) is closely related to the BP recursion, thus in bounding the fixed point of the reconstruction recursion, we hope to provide additional information on the fixed point of the BP recursion, and in turn improve the understanding of the structure of the clusters.

We conclude this section by noting the implication of our results for sampling algorithms, as non-reconstruction is closely related to the efficiency of MCMC. Typically, local algorithms are efficient only when there is no long-range correlation. Recently, it was shown that Glauber dynamics of $k$-colouring model on $d$-ary trees has $O(n\log n)$ mixing time in the entire non-reconstruction regime \cite{sly2014glauber}. Much less is known on random graphs (Erd\H{o}s-R\'{e}nyi, random $d$-regular graph, etc.). The best bounds for efficient algorithms so far are $k\ge5.5d$ using the Glauber dynamics  are \cite{efthymiou2014mcmc} and $k\ge3d$ using non-MCMC methods \cite{yin2015sampling}, both of which are still below the uniqueness threshold.

\subsection{Outline of the proof}

The proof of Theorem~\ref{thm:col_main_rec} essentially follows from a detailed analysis of the tree recursion. We begin by specifying the distribution of the reconstruction probability $P(\sigma_{\rho}=\cdot\mid\sigma_{L_{n}})$ on $n$-level trees as a function of the distribution on $(n-1)$-level trees $P(\sigma_{\rho}=\cdot\mid\sigma_{L_{n-1}})$. This defines a distributional recursion on the set of probability measures on the $k$ dimensional simplex $\Delta^{k}$. For the purpose of proving reconstruction, it is enough to show that the recursion has a non-trivial fixed point, which is done in two steps: First we show that there exists a non-trivial measure $\mu$ on $\Delta^k$ such that after one step of the recursion the new measure stochastically dominates the original one. This step is done in Section~\ref{sec:col_stoc}. Given the result of stochastic dominance, we provide a randomized algorithm such that the distribution of the reconstruction probability equals $\mu$ on trees of any depth, which is done in Section~\ref{sec:rec}.

\section{Reconstruction algorithm\label{sec:rec}}

We begin by introducing the notations we will be using throughout the proof.  In general, we will use $U,V\dots$ for random variables and $\mu,\nu$ for measures. To avoid complicated subscripts, we will use both $U$ and $\mu_{U}$ for the distribution of $U$ and use $f_{U}$ for its density (using delta functions for atoms). For any function $\phi$, we write $\phi\circ\mu$ for the distribution of $\phi(X)$, where $X$ is a random sample of $\mu$, denoted as $X\sim\mu$.  We will use $B\oplus C$ to denote the (measure of) the sum of two independent copies of $B$ and $C$, and $a\otimes B$ to denote the sum of $a$ i.i.d.\ copies of $B$. One should distinguish these two operators with $+$ and $\cdot$, the usual addition and scaler multiplication of measures. By definition, we have
\[
\mu_{B\oplus C}=\mu_{B}*\mu_{C},
\quad
\mu_{a\otimes B}=\underbrace{\mu_{B}*\mu_{B}*\cdots*\mu_{B}}_{a\text{ times}}.
\]
For any space $\Omega$, we will use $\mathcal{M}(\Omega)$ to denote
the space of probability measures on $\Omega$. A substantial portion
of our proof will be comparing different measures. For that sake,
we define the following partial order on $\mathcal{M}(\eRR)$, where
$\eRR \equiv \mathbb{R}\cup\{-\infty,\infty\}$ is the extended real numbers.
\begin{defn}[Stochastic dominance]
    For $\mu,\nu\in\mathcal{M}(\eRR)$, we say that \textbf{$\nu$ }\emph{stochastically dominates $\mu,$} denoted by $\mu\prec\nu$, if for any $x\in\eRR$, $\mu([-\infty,x])\ge\nu([-\infty,x]).$ Moreover, for any $\epsilon>0$, we say that \textbf{$\nu$} \emph{stochastically dominates $\mu$ by $\epsilon$}\textbf{, }denoted by $\mu\prec_{\epsilon}\nu$, if for any $x\in\eRR$, we have either $\mu([-\infty,x])=1$ or $\nu([-\infty,x])=0$ or $\mu([-\infty,x])-\epsilon\ge\nu([-\infty,x])$.
\end{defn}
The following proposition gives two sufficient conditions of stochastic dominance that will be used throughout the proof. The proof of proposition should be trivial.
\begin{prop} \label{prop:rec_stocfact}
Let $X,Y$ be two arbitrary independent random variables
\begin{enumerate}
\item If $\mu_{X},\mu_{Y}$ are absolutely continuous and $f_{X}(y)\le f_{Y}(y)$ for all $y$ satisfying $P(Y\ge y)>0$, then $X\succ Y$.
\item If $X$ stochastically dominates $Y$ by $\epsilon$, then for any random variable $X'$ such that $P(X\neq X')\le\epsilon$ and $\{x';P(X'<x')=0\}\subseteq\{y;P(Y<y)=0\}$, $X'$ also stochastically dominates $y$.
\end{enumerate}
\end{prop}

\subsection{\texorpdfstring{$k$}{k}-colouring model and the tree recursion}

In this section we give the distributional recursion involved in the reconstruction problem.  Let $[k]=\{1,\dots,k\}$ denote the set of $k$-colours and $T=(V,E)\sim \mathcal{T}_\xi$ be an instance of the Galton-Watson tree of offspring distribution $\xi$ with root $\rho$.
For each $n\ge 1$, let $T_n=(V_n,E_n)$ denote the restriction of $T$ to its first $n$ levels and let $L_{n}$ be the leaves of $T_n$.
For each $n$, the $k$-colouring model restricted on $T_n$ is the uniform measure on the set of proper colourings 
$$
\Omega_n:=\set{
    \sigma\in[k]^{V_n}:\sigma_{u}\neq\sigma_{v},
    \textup{ for all } e=(u,v)\in E_n}.
$$  
And we will use $\Omega(L_{n})$ to denote the set of possible configurations on $L_{n}$.

For any $\eta\in\Omega(L_{n})$ and $l\in[k]$, let $f_n$ be the (deterministic) function defined as follows:
\[
f_{n}(l,\eta;T):=P(\sigma_{\rho}=l|T_n, \sigma_{L_{n}}=\eta).
\]
Given tree $T_n$ and the observed configuration $\eta\in\Omega(L_{n})$, the maximum likelihood estimator of $\sigma_{\rho}$ is the colour $l$ that achieves the maximum of $f_{n}(l,\eta;T)$, and this estimation is correct with probability $\max_{l}f_{n}(l,\eta;T)$. Let $d_\rho$ be the degree of the root $\rho$ of $T$, and $u_1,\dots,u_{d_\rho}$ be the $d_\rho$ offspring of the root $\rho$. For each $1\le i\le d_\rho$, let $T_i$ be the subtree rooted at $u_i$ and $L_{i}^{n}=L_{n}\cap T_{i}$ be the subset of $L_{n}$ restricted to $T_{i}$. Given the colour of $u_{i}$, the configuration on $T_{i}$ is independent of the configuration on $T\backslash T_{i}$.  A standard recursive calculation gives that, for each $\eta\in\Omega(L_{n})$
and $l\in[k]$,
\begin{equation}
    f_{n+1}(l,\eta;T)
    =\frac{\prod_{i=1}^{d_{\rho}}
        (1-f_{n}(l,\eta_{i},T_i)}
    {\sum_{m=1}^k \prod_{i=1}^{d_{\rho}}
        (1-f_{n}(m,\eta_{i};T_i))}
    .\label{eq:col_rec_treerec0}
\end{equation}

To study one step of the recursion from a vertex, one first samples the number of offspring from $\xi$ then decides the colour of each offspring accordingly.  Let $\Xi^l=\Xi^l(n;\xi)$ denote the distribution of $(T_n,\sigma_{L_{n}})$ given $\sigma_{\rho}=l$ and let $(T_n,\eta^{l})$ be a sample from $\Xi^l$.  Then the vector of posterior probability $\vX_{n}:=(f_{n}(1,\eta^{1};T),\dots,f_n(k,\eta^{1};T))$ is a random vector in the $k$-dimensional simplex $\Delta^{k}:=\{(x_{1},\dots,x_{n}):x_{i}\ge0,\sum_{i=1}^{n}x_{i}=1\}$.
Let $(T_i,\eta_{i}^{l})$ be the restriction of $(T_n,\eta^l)$ onto $T_{i}$. By the symmetry between branches of Galton-Watson trees and the symmetry between colours,
we have that
\[
(f_{n}(m,\eta^{l};T))_{m=1}^{k}\deq(X_{n}^{(m-l+1)})_{m=1}^{k},
\]
where we uses the notation $x^{(l)}$ to denote the $l$-th entry of vector $\vec{x}$, modulo $k$ when necessary.
Furthermore, conditioned on the value of $\vX_{n}^{(1)}$, $(\vX_{n}^{(2)},\dots,\vX_{n}^{(k)})$ are exchangeable. In particular $\vX_{n}^{(l)}\deq\vX_{n}^{(2)}$ for all $l\neq1$.

The distribution of $\vX_n$ can be solved recursively using the following $\Delta^k$-valued function $\Gamma$ that takes an indefinite number of variables: Let
\begin{equation}
\Gamma^{(m)}(\vec{x}_{i,l},l=1,\dots k,i=1,\dots b_{l})
:=\frac{
    \prod_{l=2}^{k} \prod_{i=1}^{b_{l}}
    (1-\vec{x}_{i,l}^{(m-l+1)})
}{
    \sum_{l'=1}^{k} \prod_{l=2}^{k}\prod_{i=1}^{b_{l}}
    (1-\vec{x}_{i,l}^{(l'-l+1)})}
    ,\ \forall m\in[k],
    \label{eq:col_rec_defGamma}
\end{equation}
where we adopt the convention of $\prod_{i\in\varnothing}a_{i}=1$.
Here $b_{l}$ represent the number of $u_i$'s with colour $l$. Given $d_\rho$ and $\sigma_{\rho}=1$, the joint distribution of $(b_{2},\dots,b_{k})$ follows the multinomial distribution with sum $d_{\rho}$ and probability $(\frac{1}{k-1},\dots,\frac{1}{k-1})$ and $b_1=0$.
Let $D_\rho,(B_1,\dots,B_k)$ be an i.i.d.\ copy of $d_\rho,(b_1,\dots,b_k)$ and $\vX_{i,l}$ be i.i.d.\ samples of $\vX_{n}$, (\ref{eq:col_rec_treerec0}) implies that
\begin{equation}
    \vX_{n+1}
    \deq
    \left(
    \frac{\prod_{l=2}^{k}\prod_{i=1}^{B_{l}}
        (1-\vX_{i,l}^{(m-l+1)})}
    {\sum_{m'=1}^{k}\prod_{l=2}^{k}\prod_{i=1}^{B_{l}}
        (1-\vX_{i,l}^{(m'-l+1)})}
    \right)_{m=1}^k
    =\Gamma(\vX_{i,l},l=1,\dots k,i=1,\dots B_{l})
    .\label{eq:col_rec_treerec}
\end{equation}
Let $\tilde{\Xi}$ be the distribution of $(T_n,\sigma_{L_n})$ without conditioning on the value of $\sigma_\rho$ and define the unconditional posterior probability $\tilde{X}_{n}:=(f_{n}(1,\tilde{\eta};T),\dots,f_{n}(k,\tilde{\eta},T))$ similarly, where $\tilde{\eta}$ is sampled from $\tilde{\Xi}$. The distribution of $\vX_{n}$ and $\tilde{X}_{n}$ satisfies that at each point $x\in\Delta^k$,
\begin{align}
P(\vX_{n}\in dx)
& = kP\left(
    \sigma_{\rho}=1,
    \big(P(\tau_{\rho}=j\mid T_n,\tau_{L_{n}}=\sigma_{L_{n}})
    \big)_{j=1}^{k}\in dx
\right)\nonumber \\
 & =kP(\tilde{X}_{n}\in dx)
    P(\sigma_{\rho}=1\mid
    \big(
        P(\tau_{\rho}=j\mid T_n, \tau_{L_{n}}=\sigma_{L_{n}})
    \big)_{j=1}^{k}\in dx)
\nonumber \\
 & =kx^{(1)}P(\tilde{X}_{n}\in dx)
.\label{eq:col_rec_nutonuplus}
\end{align}

Equation (\ref{eq:col_rec_treerec}) and (\ref{eq:col_rec_nutonuplus}) are all we need to describe the distributional recursion. To be more concrete, we introduce some further notations.
Let $\mathcal{M}_{s}(\Delta^{k})\subset\mathcal{M}(\Delta^{k})$ be the subset of measures in $\mathcal{M}(\Delta^{k})$ that are invariant under permutations of the coordinates. With some abuse of notation, we will also use $\Gamma$ for the transformation it induces on $\mathcal{M}(\Delta^{k})$, i.e.\ for any $\nu\in\mathcal{M}(\Delta^{k})$, we define $\Gamma\nu$ as the distribution of $\Gamma(\vX_{i,l},l=1,\dots,k,i=1,\dots B_{l})$ where $\vX_{i,l}$ are i.i.d.\ copies with distribution $\nu$ and $B_l$ are defined as before.
For each $\nu\in\mathcal{M}_{s}(\Delta^{k})$, let $\Pi_l\nu$ be defined as  $(\Pi_l\nu)(dx):=kx^{(l)}\nu(dx)$ and define 
\begin{equation}
    \Gamma_{s}\nu
    :=\frac{1}{k}\sum_{l=1}^{k}
        (\Gamma\circ\Pi_l)\nu
.\label{eq:defGammas}
\end{equation} 
Under these notations, if $\tilde{X}_{n}\sim\nu$, then
$
\vX_{n}\sim\Pi_1\nu,\
\vX_{n+1}\sim\Gamma\circ \Pi_1\nu
\textup{ and }\tilde{X}_{n+1}\sim\Gamma_{s}\nu.
$

It is easy to check that $\delta_{(\frac{1}{k},\dots,\frac{1}{k})}$ is a trivial fixed point of $\Gamma_{s}$, which corresponds to no information about the root. To show reconstruction, it is enough to prove for $\tilde{X}_{0}\sim\mu_{0}:=\frac{1}{k}[\delta_{(1,0,\dots,0)}+\cdots+\delta_{(0,\dots,0,1)}]$ that $\Gamma_{s}^{n}\mu_{0}$ is weakly bounded away from $\delta_{(\frac{1}{k},\dots,\frac{1}{k})}$.  One of the main difficulties for analyzing graph colourings is that the dimension of the recursion  grows linearly in $k$. Luckily, as it will become clear in the proof, it is sufficient to consider only the largest coordinate of $\tilde{X}_{n}$. All the other entries are w.h.p.\ negligible as $k\to\infty$. Since we are not aiming at the tightest possible bound, we shall discard this extra information  reducing the recursion to $\mathbb{R}$.

Define $\lambda(\vec{x})=(\lambda^{(0)},\lambda^{(1)})(\vec{x}):=(\|\vec{x}\|_{\infty},\arg\max\vec{x})$
and $\Lambda:\Delta^{k}\to\Delta^{k}$ to be
\begin{equation}
\Lambda^{(m)}(\vec{x})=\begin{cases}
\|\vec{x}\|_{\infty} & m=\arg\max\|\vec{x}\|_{\infty}\\
\frac{1-\|\vec{x}\|_{\infty}}{k-1} & \textup{otherwise}
\end{cases}.
\label{eq:Lambda}
\end{equation}
We are mostly interested in the transformation $\lambda$ and $\Lambda$ induces on spaces of probability measures. With some abuse of notation, we allow extra randomness to be used to break ties in the $\arg\max$ of $\lambda$ and $\Lambda$ independently and uniformly randomly. For example if
$X=(\frac{1}{2},\frac{1}{2},0,\dots,0)$ with probability $1$, then $\lambda(X)$ equals $(\frac{1}{2},1)$ or $(\frac{1}{2},2)$ with probability $\frac{1}{2}$. Let $\Lambda^{k}=\Lambda(\Delta^{k})\subset\Delta^{k}$ be the ``star-shaped'' image of $\Lambda$, $\lambda(\vec{x})$ gives a bijection between $\Lambda^{k}\backslash(\frac{1}{k},\frac{1}{k},\dots,\frac{1}{k})$ and $(\frac{1}{k},1]\times[k]$. Hence there is a bijection between
$\mathcal{M}([\frac{1}{k},1])$ and $\mathcal{M}_{s}(\Lambda^{k}):=\mathcal{M}_{s}(\Delta^{k})\cap\mathcal{M}(\Lambda^{k})$
given by:
\begin{align*}
\lambda^{(0)}:\mathcal{M}_{s}(\Lambda^{k})\to\mathcal{M}([\frac{1}{k},1]),
\quad & \mu\to\lambda^{(0)}\circ\mu=\|\mu\|_{\infty};\\
\lambda^{-1}:\mathcal{M}([\frac{1}{k},1])\to\mathcal{M}_{s}(\Lambda^{k}),
\quad & \mu\to\lambda^{-1}\circ\left(\mu\otimes\frac{1}{k}(\delta_{1}+\cdots+\delta_{k})\right).
\end{align*}
Thus $\Lambda\circ\Gamma_{s}$ induces a transformation on $\mathcal{M}_{s}(\Lambda^{k})$ and $\lambda^{(0)}\circ\Lambda\circ\Gamma_{s}\circ\lambda^{-1}=\|\Gamma_{s}\circ\lambda^{-1}\|_{\infty}$ induces a transformation on $\mathcal{M}([\frac{1}{k},1])$. With another abuse of notation, we will use the same notation for both $\mu\in\mathcal{M}([\frac{1}{k},1])$
and its unique correspondence in $\mathcal{M}_{s}(\Lambda_{k})$ and use $\Lambda\circ\Gamma_{s}$ for both transformations. Also for $\mu,\nu\in\mathcal{M}_{s}(\Lambda_{k})$,
we say $\mu\prec\nu$ iff $\mu\prec\nu$ as elements of $\mathcal{M}([\frac{1}{k},1])$.

The main technical result of this paper is the following theorem, which
will be proved in Section~\ref{sec:col_stoc}.
\begin{thm}
\label{thm:col_rec_stocdom}
There exist $\beta^{0}<1,c>0$ such that for any $k>k_{0}$,
$
    d \ge k(\log k +\log\log k + \beta^0), 
$
and $T\sim\mathcal{T}_{\pois(d)}$,
one can constructs $\mu_{k}\in\mathcal{M}([\frac{1}{k},1])$
such that $(\Lambda\circ\Gamma_{s})\mu_{k}$ stochastically dominates
$\mu_{k}$ by $c/\log k$.
\end{thm}

Using the fact that $\|\Lambda(\vec{x})\|_{\infty}=\|\vec{x}\|_{\infty}$, Theorem~\ref{thm:col_rec_stocdom} is equivalent to the statement that $\|\Gamma_{s}\mu_{k}\|_{\infty}$ stochastically dominates $\mu_{k}$ by $c/\log k$. It follows that if at some level we can reconstruct the root with success probability $\|\tilde{X}_{n}\|_{\infty}$ for some $\tilde{X}_{n}\sim\mu_{k}\in\mathcal{M}_{s}(\Lambda^{k})$,
then in the level above we can do strictly better with success probability $\|\tilde{X}_{n+1}\|_{\infty}\succ\|\tilde{X}_{n}\|_{\infty}$. However this does not directly imply reconstruction due to two reasons.
First, the proof of Theorem~\ref{thm:col_rec_stocdom} depends heavily on the low-dimensional structure of $\mu_{k}\in\mathcal{M}_{s}(\Lambda^{k})$, but in general after one step $\Gamma_{s}\mu_{k}$ no longer belongs to $\mathcal{M}_{s}(\Lambda_{k})$.
Secondly, due to the non-linearity of $\Lambda\circ\Gamma_{s}$, it is not clear whether $(\Lambda\circ\Gamma_{s})\mu_{k}\succ\mu_{k}$ would imply $(\Lambda\circ\Gamma_{s})^{2}\mu_{k}\succ(\Lambda\circ\Gamma_{s})\mu_{k}$.
We address both problems in next subsection by intentionally manipulating the observed configuration and thus manually maintaining a nontrivial fixed point for the ``manipulated recursion''.

\subsection{Manipulating the tree recursions}
In this section we provide a reconstruction algorithm such that its estimator of $\sigma_\rho$ satisfies a modified recursion with the fixed point $\mu_k$ defined in Theorem~\ref{thm:col_rec_stocdom}.
Let $\SS_k$ be the symmetric group of degree $k$. For any $\pi\in \SS_k$, $\eta\in\Omega(L_{n})$ and $X\in \Delta^k$, define $\pi\circ\eta\in\Omega(L_{n})$ to be the configuration specified by $(\pi\circ\eta)_{v}=\pi(\eta_{v})$ and $\pi\circ X\in \Delta_k$ to be the vector with $(\pi \circ X)^{(l)} = X^{(\pi(l))}$.
We first illustrate the main idea with an example:

Suppose that two people, Alice and Bob, are trying to reconstruct $\sigma_\rho$, the colour of the root, from $\sigma_{L_n}$. 
Observing $T$ and  $\sigma_{L_n} = \eta\in\Omega(L_{n})$, Bob knows that root $\rho$ has colour $l$ with probability $f_{n}(l,\eta;T)$. Then Alice tells Bob that the $\eta$ he observed was not the actual $\sigma_{L_n}$, but the $\sigma_{L_n}$ after a randomly selected permutation $\pi$. Namely, $\eta = \pi\circ \sigma_{L_n}$ where $\pi$ is sampled from some distribution $\nu\in \MM(\SS_k)$.
Let $F(\eta):= (f_n(\ell,\eta;T))_{l=1}^k\in\Delta^k$ be the original estimator of the root with $T$ omitted for brevity. Bob's estimation of $\sigma_\rho$ after Alice's permutation becomes 
\begin{align*}
F(\eta;\nu)
&:=
\Big(
    P_{\pi\sim\nu}
    (\sigma_{\rho} = l \mid \pi\circ\sigma_{L_n} = \eta)
\Big)_{l=1}^k
=\sum_{\pi\in \SS_{k}}\nu(\pi)F(\pi^{-1}\circ\eta)
=\sum_{\pi\in \SS_{k}}\nu(\pi)(\pi\circ F)(\eta).
\end{align*}
Thus if Alice chooses the distribution $\nu$ carefully, she can manipulate Bob's estimation to any vector in the convex hull of $\big\{(\pi\circ F)(\eta):\pi\in \SS_k\big\}$. And that's essentially what we will do in this section.
In particular, we consider the following two families of $\nu\in\MM(\SS_{k})$:
\begin{enumerate}
\item For each $l\in[k]$, let $\nu_{1}(l)$ be the uniform distribution on $\SS_{[k]\setminus l } := \set{\pi\in \SS_k: \pi_l=l}$. For any $\eta\in\Omega(L_{n})$
and $m\in[k]$, 
\begin{equation}
    F^{(m)}(\eta;\nu_1(l))
=\begin{cases}
    f_{n}(m,\eta) & m=l\\
    \frac{1}{k-1}\sum_{m\neq l}f_{n}(m,\eta) & m\ne l
\end{cases}
=\begin{cases}
    f_{n}(m,\eta) & m=l\\
    \frac{1}{k-1}\left(1-f_{n}(m,\eta)\right) & m\ne l
\end{cases}.
\label{eq:rec_defnu1}
\end{equation}

\item For each $p\in[0,1]$, let $\nu_{2}(p):=p\nu_{\mathrm{unif}}+(1-p)\delta_{\mathrm{id}}$
where $\nu_{\mathrm{unif}}$ is the uniform distribution on $\SS_{k}$
and $\delta_{\mathrm{id}}$ is the point mass at the identity permutation $\mathrm{id}$.
For any $\eta\in\Omega(L_{n})$,
\begin{equation}
    F(\eta;\nu_2(p))
=(1-p)F(\eta)
+\frac{p}{k!}\sum_{\pi\in \SS_k}(\pi\circ F)(\eta)
=(1-p)F(\eta)
+p\cdot \Big(\frac{1}{k},\dots,\frac{1}{k}\Big)
.\label{eq:rec_defnu2}
\end{equation}
\end{enumerate}
In the proof, we will use $\nu_1(l)$ to simulate the transformation $\Lambda$ defined in \eqref{eq:Lambda} and $\nu_2(p)$ to reduce the distribution $(\Lambda\circ\Gamma_s) \mu_k$ to $\mu_k$. For the later purpose, we show the following lemma.
\begin{lem} \label{lem:col_rec_reduce}
For any $\mu_{1},\mu_{2}\in\mathcal{M}([\frac{1}{k},1])$ such that $\mu_{1}\succ\mu_{2}$, there exist function $q:[\frac{1}{k},1]\times[0,1]\to[\frac{1}{k},1]$, such that $q(y,u)\le y$ for all $y\in[\frac{1}{k},1]$,$u\in[0,1]$ and for any independent random variables $Y\sim\mu_{1}$ and $U\sim\mathrm{Unif}[0,1]$,
$q(Y,U)\sim\mu_{2}$. We say that such function $q$ reduces $\mu_{1}$
to $\mu_{2}$.
\end{lem}
\begin{proof}
Let $G_{1}$, $G_{2}$ be the c.d.f.\ of $\mu_{1},\mu_{2}$, and
$G_{1}(x-0)$ be the left limit of $G_{1}$ at $x$. For $y\ge\frac{1}{k}$,
define
\[
q(y,u):=\inf\Big\{x\ge\frac{1}{k}:G_{2}(x)\ge G_{1}(y-0)+u(G_{1}(y)-G_{1}(y-0))\Big\}.
\]
Note that $\mu_{1}\succ\mu_{2}$ implies that $G_{2}(y)\ge G_{1}(y)$
for all $y\ge\frac{1}{k}$. Hence $q(y,u)\in[\frac{1}{k},y]$. Let
$y_{x}=\sup\{y:G_{1}(y-0)\le G_{2}(x)\}$. A direct calculation shows
that for $x\ge\frac{1}{k}$,
\begin{align*}
P(q(Y,U)\le x) & =P(G_{2}(x)\ge G_{1}(Y-0)+U(G_{1}(Y)-G_{1}(Y-0)))\\
 & =G_{1}(y_{x}-0)+(G_{1}(y_{x})-G_{1}(y_{x}-0))\frac{G_{2}(x)-G_{1}(y_{x}-0)}{G_{1}(y_{x})-G_{1}(y_{x}-0)}=G_{2}(x).
\end{align*}
\end{proof}
Recalling the 1-to-1 correspondence between $\MM(\Lambda^k)$ and $\MM([\frac{1}{k},1])$, we define $q_0$ to be the function that reduces $\mu_{0}=
\frac{1}{k}(\delta_{(1,0,\dots0)}+\cdots+\delta_{(0,\dots,0,1)})$ to $\mu_k$ and $q_\star$ to be the function that reduces $(\Lambda\circ \Gamma_s) \mu_k$ to $\mu_k$, where the later one exists because $(\Lambda\circ \Gamma_s) \mu_k \succ \mu_k$. We further define for each $\bullet\in\set{0,\star}$ that 
\begin{equation}
    \tq_{\bullet}(y,u) := \frac{ky - q_\bullet(y,u)}{ky-1}
    \in[0,1]
    \quad \textup{such that }\ 
    (1-\tq_\bullet(y,u)) \cdot y 
        + \tq_\bullet(y,u) \cdot \frac{1}{k} 
    = q_\bullet(y,u).
\label{eq:tq}
\end{equation}

Let us introduce further notations necessary for the algorithm:
Let $\iU:=(U_v)_{v\in T}$ be an array of independent $\mathrm{Unif}[0,1]$ random variables indexed by the vertices of $T$ and let 
$\iU_v :=(U_w)_{w\in T_v}$ be the sub-array indexed over $T_v$, 
the subtree rooted at $v$. 
For each $v\in T$ and $w\in T_v$, we will encode Alice's action on $T_v$ and Bob's information at $w$ after Alice's actions on $T_v$ as 
$$
\ia_v:=(p_v,l_v,\pi_v)\in [0,1]\times [k]\times \SS_k
\quad\textup{and}\quad
\ib_{w,v}:=(p_{w,v},\eta_{w,v})\in [0,1]\times [k].
$$
Let $\iA_v$ and $\iB_v$ be arrays of $\ia_w$ and $\ib_{w,v}$ indexed over $w\in T_v$ respectively. Letting $L^v_1$ denote the set of offspring of $v$, we define $\iB_{L^v_1}:= (\ib_{w,u})_{u\in L^v_1,w\in T_u}$ as the concatenation of $(\iB_u)_{u\in L^v_1}$ for each $v\notin L_n$ and define $\iB_{L^v_1} := (\sigma_v)$ otherwise. With the meaning of $\ia_v$ and $\ib_{w,v}$ to be given in a moment, we formally define
\begin{align*}
    \iP^\circ_{v} 
    &:= 
    \iP^\circ_{v} (\iB_{L^v_1}) =
    \begin{cases}
    (P(\sigma_v=l\mid \sigma_v))_{l=1}^k 
    &  v\in L_n.
    \\
    (
        P(\sigma_{v} = l \mid 
            \iB_{L^v_1}
        )
    )_{l=1}^k
    & v \notin L_n.
    \end{cases}
    ,\quad
    \iP_{v} := 
    \iP_{v}(\iB_v) = (
        P(\sigma_{v} = l \mid 
        \iB_{v}
        )
    )_{l=1}^k
    ,
\end{align*}
as Bob's belief on $\sigma_v$ before and after Alice's actions on $T_v$ (if he is given $\iB_{L_1^v}$ or $\iB_v$ respectively).

We now define the actions of Alice, namely what $\ia_v,\ib_v$ means and how she recursively constructs them from the leaves up to the root as a function of $T_v$, $\sigma_{T_v\cap L_n}$ and $\iU_v$:
\begin{enumerate}
\item For each leaf vertex $v \in L_n$, $T_v = \set{v}$. Bob's belief before Alice's action is simply
\[
    \iP^\circ_{v} 
    = (P(\sigma_v=l\mid \sigma_v))_{l=1}^k 
    = (\Ind{\sigma_v=l})_{l=1}^k
    .
\]
Alice then sets $l_v = \sigma_v$, $p_v = \tq_0(1,U_v)$ and $\pi_v = \pi^2_v\circ \pi^1_v$, where $\pi^1_v$ is a sample of $\nu_1(l_v)$ and $\pi^2_v$ is an independent sample of $\nu_2(p_v)$.
Finally, she permute $\sigma_v$ by $\pi_v$ (which has the same effect as using $\pi^2_v$) and prepares Bob's share of information as $\iB_{v} = (\ib_{v,v})$, where
\[
    \ib_{v,v} = (q_{v,v}, \eta_{v,v}) 
    =(p_v,\pi^2_v(l_v))
    =(p_v,\pi^2_v(\sigma_v))
    .
\]
\item  Suppose that for each $w\in L_{m+1}$, Alice has recorded her actions on $T_w$ as $\iA_w$  and prepared the information for Bob as $\iB_w$, where $\iA_w$ is a function of $(T_w, \sigma_{T_w\cap L_n}, \iU_w)$ and $\iB_w$ is a function of $\iA_w$. We now describe Alice's actions on $T_v$, namely how she constructs $\iA_v$ and $\iB_v$ for each $v\in L_m$ as a function of  $(\iB_u)_{u\in L^v_1}$ and $U_v$.
\begin{enumerate}
    \item First, for each $u\in L^v_1$, Alice calculates $\iP_u$, namely Bob's belief of $\sigma_{u}$ given information $\iB_{u}$.
Given $(\iP_u)_{u\in L^v_1}$, Alice calculates Bob's belief of $\sigma_v$ before her actions on $T_v$.
Following a similar recursion of \eqref{eq:col_rec_treerec0},
\[
    \iP^\circ_{v}
    =
    \Bigg(
    \frac{\prod_{u\in L^v_1}
        (1-\iP_{u}^{(l)})}
    {\sum_{m=1}^k \prod_{u\in L^v_1}
        (1-\iP_{u}^{(m)})}
    \Bigg)_{l=1}^k
    .
\]
\item  Let $U^i_v,i=1,2,3$ be three independent $\textup{Unif}[0,1]$ random variables constructed from $U_v$. Let $l_v=l_v(\iP^\circ_{v},U^1_v)$ be uniformly picked from $\set{l: (\iP^\circ_{v})^{(l)}=\|\iP^\circ_{v}\|_\infty}$, the set of largest coordinates of $\iP^\circ_{v}$, using the randomness of $U^1_v$ and let $p_v = \tq_\star(\|\iP^\circ_{v}\|_\infty, U^2_v)$.
Alice then uses the randomness $U^3_v$ to sample $\pi^1_v$ from $\nu_1(l_v)$ and $\pi^2_v$ from $\nu_2(p_v)$ independently and sets $\pi_v = \pi^2_v\circ\pi^1_v$. This gives $\ia_v = (p_v,l_v,\pi_v)$ and completes the construction of $\iA_v$.
\item Finally, Alice ``permutes'' Bob's current observation of ${T_v\cap L_n}$ and all the previous information she prepares for Bob by $\pi_v$. This, in the language of $\iA_v$ and $\iB_v$, corresponds to setting $q_{v,v} = p_v$, $\eta_{v,v} = \pi^2_v(l_v)$ and setting for each $w\in T_v\backslash\set{v}$ that
$q_{w,v} = p_w$ and 
$$
    \eta_{w,v} 
    = \pi_v (\eta_{w,w_1})
    = \pi_{w_0}(\pi_{w_1}( \cdots \pi_{w_{r-1}}(\pi^2_{w}(l_w))\cdots)\cdots),$$
where $w_0=v,w_1\in L^v_1,\dots, w_{r-1}, w_{r} = w$ is the unique path connecting $v$ to $w$. 
This completes the definition of $\iB_v=(\ib_{w,v})_{w\in T_v}$.
\end{enumerate}
\item As a final step, Alice tells Bob the array $\iB_\rho$ as partial information of her actions, which in particular includes Bob's final observation as $(\eta_{v,\rho})_{v\in L_n}$.
We emphasis that $\iB_\rho$ is the only piece of information given to Bob. All the intermediate $\iB_v$'s exist only in Alice's deduction and remain unknown to Bob.
\end{enumerate}
The main result of the section is the following Theorem.
\begin{thm}\label{thm:AliceBob}
For any $n\ge 1$, let $T$ be a $n$-level tree sampled from $\mathcal{T}_\pois$ and $\sigma_{L_n}$ be generated by the colouring model on $T$. Let $\iU$ be a $T$-indexed array of independent $\mathrm{Unif}[0,1]$ random variables. 
If Alice performs her actions as described above, then Bob's final belief of $\sigma_\rho$ after all Alice's actions, represented as
\[
    \iP_\rho = \iP_\rho(\iB_\rho) 
    =(P(\sigma_{\rho} = l \mid 
        \iB_{\rho}
        ))_{l=1}^k
    \in\Delta^k
    ,
\]
follows the distribution of $\mu_k$.
\end{thm}

\begin{proof}
    For each permutation $\pi\in\SS_k$ and $T$-indexed array $\iB = (\ib_v)_{v\in T}\in ([0,1]\times [k])^T$, let $\pi \circ \ib_v := (p_v,\pi(\eta_v))$ and  $\pi\circ\iB := (\pi\circ\ib_v)_{v\in T}$.
    We induct on the number of levels in tree $T$ to prove the claim of Theorem~\ref{thm:AliceBob} together with the result that 
\begin{equation}
    \iP_\rho(\pi\circ \iB_\rho)
    =\Big(
    P(\sigma_\rho=l \mid
    \pi \circ \iB_\rho
    )
    \Big)_{l=1}^k
    =\pi^{-1} \circ \iP_\rho(\iB_\rho).
\label{eq:perinvar}
\end{equation}
    
For $n=0$, $T=\set{\rho}$ is the singleton tree and $\iP^\circ_\rho= (\Ind{\sigma_\rho = l})_{l=1}^k$, Bob's belief before Alice's action, follows distribution $\mu_0$. 
Given $\ib_{\rho,\rho} = (p_\rho,\eta_{\rho,\rho})$, Bob's posterior estimation of $\sigma_\rho$ satisfies
\[
    P(\sigma_\rho = \tilde{\pi}^{-1}(\eta_{\rho,\rho})
    \mid \ib_{\rho,\rho})
    = \nu_2(p_\rho) (\tilde{\pi})
    ,\quad \forall \tilde{\pi}\in\SS_k.
\]
Therefore, applying \eqref{eq:rec_defnu2}, Bob's belief of $\sigma_\rho$ after Alice's action at $\rho$ becomes
\[
    \iP_\rho = \Big(
        P(\sigma_\rho=l \mid
            \pi_{\rho}(\sigma_\rho) 
            = \eta_{\rho,\rho}
        )
    \Big)_{l=1}^k
    = (1-p_\rho) \iP^\circ_\rho 
    + p_\rho\cdot \Big(\frac{1}{k},\dots,\frac{1}{k}\Big).
\]
Observe that by definition $p_\rho = \tq_0(1,U_\rho) = \tq_0(\|\iP^\circ_\rho\|_\infty,U_\rho)$. Lemma~\ref{lem:col_rec_reduce} and \eqref{eq:tq} then imply that $\iP_\rho$ follows the distribution of $\mu_k$.
It is not hard to check that \eqref{eq:perinvar} also holds.

Suppose we have proved Theorem~\ref{thm:AliceBob} and \eqref{eq:perinvar} for trees no greater than $n-1$ levels, we now proceed to trees of $n$ levels. By the induction hypothesis, for each $u\in L_1$, $\iP_u=\iP_u(\iB_u)$, Bob's belief of $\sigma_u$ after Alice's actions on $T_u$, follows the distribution $\mu_k$. 
Following a similar calculation of \eqref{eq:col_rec_nutonuplus},
we can show that conditioning on $\sigma_\rho=l$ but not $T$ and $\sigma_{T\setminus\set{\rho}}$, $(\iP_u)_{u\in L_1}$ has the same joint distribution as $\pois(d)$ independent samples of $\Pi_{l}\mu_k$. Therefore 
\[
    \iP^\circ_\rho =
    \left(
    \frac{\prod_{u\in L_1}
        (1-\iP_{u}^{(l)})}
    {\sum_{m=1}^k \prod_{u\in L_1}
        (1-\iP_{u}^{(m)})}
    \right)_{l=1}^{k}
    \sim
    \Gamma_s \mu_k.
\]

Now we turn to $\iP_\rho = \iP_\rho(\iB_\rho)$. For each $u\in L_1$, let $\iB_{\rho,u}:=(\ib_{w,\rho})_{w\in T_u}$, $\iB_{\rho,L_1}:=(\ib_{w,\rho})_{w\in T_\rho\setminus\set{\rho}}$ be sub-arrays of $\iB_\rho$.
Using the induction hypothesis on \eqref{eq:perinvar}, for each $\pi\in\SS_k$ we have
\begin{align*}
    \iP^\circ_\rho(\pi\circ \iB_{L_1}) 
    &=
    \left(
    \frac{\prod_{u\in L_1}
    (1-\iP_{u}^{(l)}(\pi\circ \iB_{u}))}
    {\sum_{m=1}^k \prod_{u\in L_1}
    (1-\iP_{u}^{(m)}(\pi\circ \iB_{u}))}
    \right)_{l=1}^{k}
    =
    \left(
    \frac{\prod_{u\in L_1}
    (1-\iP_{u}^{(\pi^{-1}(l))}(\iB_{u}))}
    {\sum_{m=1}^k \prod_{u\in L_1}
    (1-\iP_{u}^{(m)}(\iB_{u}))}
    \right)_{l=1}^{k}
    \\&=
    \pi^{-1}\circ \iP^\circ_\rho(\iB_{L_1}) 
    .
\end{align*}
Hence   set
$ \set{l: 
(\iP^\circ_\rho (
    \tilde{\pi} \circ 
    \pi^1_\rho \circ 
    \iB_{L_1})
)^{(l)}
= \| \iP^\circ _\rho (\iB_{L_1}) \|_\infty
}
$
has the same size for all $\tilde{\pi}\in \SS_k$ and contains $l_\rho$ if $\tilde{\pi} \in \supp \nu_1(l_\rho)$.
Furthermore, by the symmetry of $\sigma_{L_n}$, each element of $\set{\pi \circ \iB_\rho}_{\pi\in\SS_k}$ is equally likely to happen.
Therefore by \eqref{eq:rec_defnu1}, the belief of Bob after the first action of Alice on $T_\rho$ satisfies that
\begin{align*}
    \iP^1_\rho =
    \iP^1_\rho(l_\rho, \pi_\rho^1 \circ \iB_{L_1})
    &:=
    \Big( P(\sigma_\rho=l\mid
        l_\rho,
        \pi_\rho^1 \circ \iB_{L_1}
    )\Big)_{l=1}^k
    \\&=
    \sum_{\tilde{\pi}\in \SS_k}
    \nu_1(l_\rho)(\tilde{\pi})
    \iP^\circ_\rho(
        \tilde{\pi}^{-1}\circ
        \pi^{1}_\rho\circ
        \iB_{L_1})
    =
    \Lambda(\iP_\rho^\circ)
    ,
\end{align*}
where the same randomness $U^1_\rho$ is used in breaking ties of $\Lambda$.
It follows that $\iP^1_\rho\sim (\Lambda\circ \Gamma_s) \mu_k$.

Next we note that for any $\tilde{\pi}\in\SS_k$, $\|\iP^\circ_\rho(\tilde{\pi}\circ \iB_{L_1}) \|_\infty = \|\iP^\circ_\rho(\iB_{L_1}) \|_\infty$.
Therefore $p_\rho$, as a function of $\|\iP^\circ_\rho(\iB_{L_1}) \|_\infty$ and $U^2_\rho$, is invariant under permutations of $\iB_{L_1}$.  Given $\ib_{\rho,\rho} = (p_\rho,\eta_{\rho,\rho})$, Bob's posterior estimation of $l_\rho$ and  $\pi^1_\rho \circ \iB_{L_1}$ satisfies that

\[
    P( l_\rho = \tilde{\pi}^{-1}_2(\eta_{\rho,\rho})
    ,
    \pi^1_\rho \circ \iB_{L_1} 
    = \tilde{\pi}^{-1}_2 \circ \iB_{\rho,L_1}
    \mid
    \iB_\rho)
    = \nu_2(p_\rho)(\tilde{\pi}_2)
    .
\]
Applying \eqref{eq:rec_defnu2}, we have that
\[
    \iP_\rho(\iB_\rho)=
    \sum_{\tilde{\pi}_2 \in \SS_k}
    \nu_2(p_\rho)(\tilde{\pi}_2)
    \iP^1_\rho(
        \tilde{\pi}^{-1}_2(\eta_{\rho,\rho}),
        \iB_{\rho,L_1}
    ) 
    = (1-p_\rho) \iP^1_\rho(
        \eta_{\rho,\rho},
        \iB_{\rho,L_1}
    )
    + p_\rho\cdot \Big(\frac{1}{k},\dots,\frac{1}{k}\Big)
    .
\]
Recall that $p_\rho = \tq_\star(\|\iP^\circ_\rho\|_\infty,U^2_\rho) = \tq_\star(\|\iP^1_\rho\|_\infty,U^2_\rho)$ where $q_\star$ is the function that reduces $(\Lambda\circ \Gamma_s)\mu_k$ to $\mu_k$ and $\tq_\star$ is defined in \eqref{eq:tq}.
Lemma~\ref{lem:col_rec_reduce} then implies that
$\iP_\rho$ follows the distribution of $\mu_k$.

Finally we finish the induction hypothesis of \eqref{eq:perinvar}.
Observe that for $\tilde{\pi}\sim \nu_1(l)$, $\pi\circ \tilde{\pi} \circ \pi^{-1} $ follows the distribution $\nu_1(\pi(l))$.
For each $\pi\in\SS_k$, we have
\begin{align*}
    \iP^1_\rho(\pi(l_\rho), \pi(\pi_\rho^1 \circ \iB_{L_1}))
    &=
    \sum_{\tilde{\pi}\in \SS_k}
    \nu_1(\pi(l_\rho))(\tilde{\pi})
    \iP^\circ_\rho(\tilde{\pi}^{-1}\circ \pi \circ\iB_{L_1})
    \\&=
    \sum_{\tilde{\pi}\in \SS_k}
    \nu_1(l_\rho)(\tilde{\pi})
    \iP^\circ_\rho(
        \pi\circ\tilde{\pi}^{-1}\circ \pi^{-1}
        \circ \pi \circ\iB_{L_1}
    )
    \\&=
    \sum_{\tilde{\pi}\in \SS_k}
    \nu_1(l_\rho)(\tilde{\pi})
    \iP^\circ_\rho(
        \pi\circ\tilde{\pi}^{-1} \circ\iB_{L_1}
    )
    = \pi^{-1} \circ \iP^1(l_\rho,\pi^1_\rho\circ \iB_{L_1})
    .
\end{align*}
It follows that
\[
    \iP_\rho(\pi\circ \iB_\rho)
    = (1-p_\rho) \iP^1_\rho(
        \pi(\eta_{\rho,\rho}),
        \pi\circ \iB_{\rho,L_1}
    )
    + p_\rho\cdot \Big(\frac{1}{k},\dots,\frac{1}{k}\Big).
    = \pi^{-1} \circ \iP_\rho(\iB_\rho).
\]
And that finishes the proof the induction hypothesis.
\end{proof}
Theorem~\ref{thm:col_rec_stocdom} and Theorem~\ref{thm:AliceBob} immediately imply the following result.
\begin{cor} \label{cor:AliceBob}
For any $d,k$ such that Theorem~\ref{thm:col_rec_stocdom} holds,
there exist independent random array $\iU$ and measurable function $\iB_\rho(T,\sigma_{L_n},\iU)$ such that
\[
    \liminf_{n\to\infty} \E \sup_{l\in[k]}
    \left|
        P\left(
            \sigma_{\rho}=l \mid
            \iB_\rho(T_n,\sigma_{L_n},\iU)
        \right)
        - \frac{1}{k}
    \right|>0.
\]
\end{cor}
\subsection{Regular trees}
The result of Theorem~\ref{thm:AliceBob} and Corollary~\ref{cor:AliceBob} can be modified to regular trees 
by, roughly speaking, truncating $T\sim\mathcal{T}_{d}$ into a smaller tree:
Let $\tpois(d',d)$ be the truncated Poisson distribution defined as the distribution of $D'\cdot \Ind{D'\le d}$ where $D'\sim \pois(d')$ and let $\mathcal{T}_{\tpois(d'd)}$ be the Galton-Watson tree of offspring distribution $\tpois(d',d)$. There exists a natural coupling between $T_1\sim\mathcal{T}_{\tpois(d'd)}$, $T_2\sim \mathcal{T}_{\pois(d')}$ and $T \sim \mathcal{T}_d$ such that $T_1$ is a subtree of $T_2$ and $T$ with probability $1$.

Recall that $\MM(\Delta^k)$-operator $\Gamma$ defined in \eqref{eq:col_rec_treerec} depends implicitly on the offspring distribution $\xi$. We differentiate the two operators under $\xi =\mathcal{T}_{\pois(d')}$ and $\xi =\mathcal{T}_{\tpois(d',d)}$ as $\Gamma^{p}$ and  $\Gamma^{t}$ respectively. 
Fix $\beta^*\in(\beta^{0},1)$. For any $d,k$ satisfying \eqref{eq:col_main_recregime}, let $d':=\lfloor d-(\beta^*-\beta^{0})k\rfloor$. For $k\ge k_0(\beta^\star,c)$, 
\begin{equation}\label{eq:dTVpt}
d_{\mathrm{TV}}(\Lambda\circ\Gamma^{p}_s\mu_{k},\Lambda\circ\Gamma^{t}_s\mu_{k})
\le P(\pois(d')>d)<c(k\log k)^{-1}.
\end{equation}
Therefore if $(d',k)$ further satisfies Theorem~\ref{thm:col_rec_stocdom}, then $(\Lambda\circ\Gamma^{t})_s\mu_{k}$ stochastically dominates $\mu_{k}$. Thus we can find function $q_t$ that reduces $(\Lambda\circ\Gamma^{t})_s\mu_{k}$ to $\mu_{k}$ and define $\tq_t$ similarly.

Let $T\sim \mathcal{T}_d$ be the $n$-level $d$-ary tree and $\iD:=(D_v)_{v\in T}$ be a $T$-indexed array of independent $\textup{tPois}(d',d)$ random variables. 
We now describe the necessary modification such that $\itA_v$, $\itB_v$, $\itP^\circ_v$, $\itP_v$ can be constructed in a similar fashion as $\iA_v, \iB_v$, $\iP^\circ_v$, $\iP_v$. The construction remains the same for each $v\in L_n$. For each $v\notin L_n$, we proceed with the following changes:
\begin{enumerate}
    \item In step 2(a), instead of considering all $u\in L^v_1$, Alice now only uses the first $D_v$ vertices and discards the rest. Namely, letting $u_1,\dots,u_d$ be the $d$ offspring of $v$, she calculates
\[
    \itP^\circ_{v}
    :=
    \Bigg(
    \frac{\prod_{i=1}^{D_v}
        (1-\itP_{u_i}^{(l)})}
    {\sum_{m=1}^k \prod_{i=1}^{D_v}
        (1-\itP_{u_i}^{(m)})}
    \Bigg)_{l=1}^k
    ,
\]
and sets $\itb_{w,v}=(\star,\star)$ for each $w\in T_{u_i}, i>D_v$. She then continues to set $\tilde{\ia}_v$ and the rest of $\itB_{v}$ using $\itP^\circ_v$ and $U_v$.
\item In step 2(b), instead of setting $p_v = \tq_\star(\|\iP^\circ_{v}\|_\infty, U^2_v)$, Alice sets $p_v = \tq_t(\|\itP^\circ_{v}\|_\infty, U^2_v)$.
\end{enumerate}
In short, Bob now has to reconstruct $\sigma_\rho$ based only on the information $\itB_\rho$ of a truncated tree of $T$ sampled from $\mathcal{T}_{\tpois(d',d)}$, as the information on the rest of the vertices are erased and set to $(\star,\star)$.

\begin{cor}
    \label{cor:rec_regulartree}Fix $\beta^\star\in(\beta^0,1)$. For any $d$, $k$ such that $(d':=\lfloor d-(\beta^*-\beta^{0})k\rfloor,k)$ satisfies Theorem~\ref{thm:col_rec_stocdom} and \eqref{eq:dTVpt},
    there exist independent random arrays $\iU,\iD$ and measurable function $\itB_\rho(\sigma_{L_n},\iU,\iD)$ such that
\[
    \liminf_{n\to\infty} \E \sup_{l\in[k]}
    \left|
        P\left(
            \sigma_{\rho}=l \mid
            \itB_\rho(\sigma_{L_n},\iU,\iD)
        \right)
        - \frac{1}{k}
    \right|>0.
\]
\end{cor}
\begin{proof}
By an essentially parallel argument of Theorem~\ref{thm:AliceBob},  we can inductively show that $\itP^\circ_v$, as a function of $(T,\sigma_{T_v\cap L_n},\iU_v,\iD_v)$, follows the distribution of $\Gamma_s^t \mu_k$ and hence $\itP_v \sim \mu_k$ for each $v\in T$.
Corollary~\ref{cor:rec_regulartree} then follows immediately.
\end{proof}

\begin{proof}[Proof of Theorem~\ref{thm:col_main_rec}]
    Let $\beta^0,c$ be the constant in Theorem~\ref{thm:col_rec_stocdom} and $\beta^*$ be selected in Corollary~\ref{cor:rec_regulartree}.  For any $k\ge k_0$ and  $d,k$ satisfying \eqref{eq:col_main_recregime},  they also satisfy the conditions of Theorem~\ref{thm:col_rec_stocdom} and Corollary~\ref{cor:AliceBob}.
Therefore if the $k$-colouring model on $T\sim\mathcal{T}_{\pois(d)}$ is not reconstructible for some $d,k$ in the same region,
then  we must have
    \[
        \limsup_{n\to\infty} \E_{T\sim\mathcal{T}_{\pois(d)}}
        [\Var(\sigma_\rho\mid 
            \iB_\rho(T_n,\sigma_{L_n},\iD)
        )]
        \le
        \limsup_{n\to\infty} \E_{T\sim\mathcal{T}_{\pois(d)}}
        [\Var(\sigma_\rho\mid T_n, \sigma_{L_n})]
        = 0
        ,
    \]
where the first step follows from the fact that $\iB_\rho=\iB_\rho(T_n, \sigma_{L_n},\iU)$ is independent of $\sigma_{\rho}$ given $\sigma_{L_n}$.
But that conflicts with the result of Corollary~\ref{cor:AliceBob}. The same confliction exists with $T\sim \mathcal{T}_d$, $\itB_\rho= \itB_\rho(\sigma_{L_n},\iU,\iD)$ and Corollary~\ref{cor:rec_regulartree}. Therefore both models are reconstructible.
\end{proof}

\section{Proof of Theorem~\ref{thm:col_rec_stocdom} \label{sec:col_stoc}}

In this section we prove the stochastic dominance result of Theorem~\ref{thm:col_rec_stocdom}. In Section \ref{sec:rec23}, we first analyse the transformation $\Gamma$ induced on $\MM(\Lambda^k)$ by (\ref{eq:col_rec_treerec}) and give a parameterized candidate of $\mu_k$. In the remaining sections, we verify that the candidate does indeed satisfy Theorem~\ref{thm:col_rec_stocdom}.

\subsection{Reformulating the recursion}\label{sec:rec23}
Recall the notations in the definition of $\Gamma\mu$ in (\ref{eq:col_rec_treerec}), where $\mu = \Pi_1 \mu_s$ for some  $\mu_s\in\MM_s(\Lambda^k)$.
For each $l\in[k], 1\le i\le B_l$, let $m_{i,l}:= m(\vX_{i,l},l) :=\arg\max_{m\in[k]} \vX^{(m-l+1)}_{i,l}$ be the coordinate of $\vX_{n+1}$ that contains the largest entry of $\vX_{i,l}$ and draw $m_{i,l}$ from $[k]$ uniformly at random if $\vX_{i,l} = (\frac{1}{k},\dots,\frac{1}{k})$.
Since $\mu$ is tilted from some symmetric measure $\mu_s$, similar to \eqref{eq:col_rec_nutonuplus},
\[
    P(m_{i,l}=m
    \mid \|\vX_{i,l}\|_{\infty}=x)=
    \Bigg\{
    \begin{array}{ll}
        x & m=l\\
        \frac{1-x}{k-1} & m\neq l
    \end{array}.
\]
Let $\mu^{=}(dx):=x\mu(dx)$ and $\mu^{\neq}(dx):=(1-x)\mu(dx)$.  The joint distribution of $(\|\vX_{i,l}\|_{\infty}, m_{i,l})$ satisfies
\[
P(\|\vX_{i,l}\|_{\infty}\in dx,m_{i,l}=m)
=
\Bigg\{
\begin{array}{ll}
    \mu^{=}(dx) & l=m\\
    \frac{1}{k-1}\mu^{\ne}(dx) & l\neq m
\end{array}
,\quad \forall x\in[0,1],\ m\in[k]
.
\]
For each $m\in[k]$, define
\[
    C_m^=:=\set{(i,m): m_{i,m} = m},\quad
    C_m^{\neq}:=\set{(i,l): l\neq m, m_{i,l} = m}
    \quad\textup{and}\quad
    C_m := C_m^= \cup C_m^{\ne}.
\]
Let $c_{m}^{=}$, $c_{m}^{\neq}$ be the cardinality of $C_{m}^{=}$ and $C_{m}^{\neq}$ respectively and set $p_{\neq}:=\mu^{\neq}([\frac{1}{k},1])=1-\mu^{=}([\frac{1}{k},1])$ to be the probability of $\{(i,l)\notin C_{l}^{=}\}$.
Note that no offspring of the root has colour $1$. Given $d_\rho=\sum_{l=1}^k B_l$, $(c_{1}^{=},c_{2}^{=},\dots,c_{k}^{=},c_{1}^{\ne},c_{2}^{\neq},\dots,c_{k}^{\neq})$
follows multinomial distribution of sum $d_{\rho}$ and probability
\begin{equation}
\frac{1}{k-1}\left(0,1-p_{\ne},\dots,1-p_{\ne},p_{\neq},\frac{k-2}{k-1}p_{\neq},\dots,\frac{k-2}{k-1}p_{\neq}\right).\label{eq:col_rec_cmrate}
\end{equation}

We now use the new notations to rewrite \eqref{eq:col_rec_treerec}. For each $\vX_{i,l}\neq(\frac{1}{k},\dots,\frac{1}{k})$, the entries of $\vX_{i,l}$ take only two values: $\|\vX_{i,l}\|_{\infty}$ and $(1-\|\vX_{i,l}\|_{\infty})/(k-1)$. And $\vX_{i,l}^{(m-l+1)}=\|\vX_{i,l}\|_{\infty}$ if and only if $m=m_{i,l}$.  Let
$
    \phi(x):=\log\big[(1-\frac{1-x}{k-1})/(1-x)\big],
$
which is an increasing function mapping $[0,1]$ to $[-\infty,\infty]$. By taking out the common factor of $\prod_{l,i}(1-\frac{1-\|\vX_{i,l}\|_{\infty}}{k-1})$, we rewrite (\ref{eq:col_rec_treerec}) as
\begin{align}
\vX_{n+1}^{(m)}
&\deq
\frac{\prod_{(i,l)\in C_{m}}(1-\|\vX_{i,l}\|_{\infty})/(1-\frac{1-\|\vX_{i,l}\|_{\infty}}{k-1})}
{\sum_{m'=1}^{k}\prod_{(i,l)\in C_{m'}}(1-\|\vX_{i,l}\|_{\infty})/(1-\frac{1-\|\vX_{i,l}\|_{\infty}}{k-1})}
=\frac{\prod_{(i,l)\in C_{m}}e^{-\phi(\|\vX_{i,l}\|_{\infty})}}
{\sum_{m'=1}^{k}\prod_{(i,l)\in C_{m'}}e^{-\phi(\|\vX_{i,l}\|_{\infty})}}
.
\label{eq:col_rec_treerec2}
\end{align}
Note that the exact value of $m_{i,l}$ when $\vX_{i,l} = (\frac{1}{k},\dots,\frac{1}{k})$ does not matter since $\phi(\frac{1}{k}) = 0$.
We further rewrite (\ref{eq:col_rec_treerec2}) as
\begin{equation}
\vX_{n+1}^{(m)} \deq
\frac{\left(
    \prod_{i=1}^{c_{m}^{=}}\exp(-\phi(Y_{i,m}^{=}))
    \prod_{i=1}^{c_{m}^{\neq}}\exp(-\phi(Y_{i,m}^{\neq}))
\right)} {
    \sum_{l=1}^{k} \left(
        \prod_{i=1}^{c_{l}^{=}}\exp(-\phi(Y_{i,l}^{=}))
        \prod_{i=1}^{c_{l}^{\neq}}\exp(-\phi(Y_{i,l}^{\neq}))
    \right)
}
=:\frac{\exp(-Z_m)}{\sum_{m=1}^k \exp(-Z_m)}
.
\label{eq:col_rec_treerec3}
\end{equation}
where $Y_{i,l}^{=}$ and $Y_{i,l}^{\neq}$ are i.i.d.~samples of $\frac{1}{1-p_{\neq}}\mu^{=}$ and $\frac{1}{p_{\neq}}\mu^{\neq}$ respectively and
\[
Z_{m}:=\sum_{i=1}^{c_{m}^{=}}\phi(Y_{i,m}^{=})+\sum_{i=1}^{c_{m}^{\neq}}\phi(Y_{i,m}^{\neq}).
\]
We conclude our calculation so far in the following claim.
\begin{prop}\label{prop:col_rec_rewrite}
    For any $d,k$, if there exists $\nu_k\in\MM([\frac{1}{k}, 1])$ (with its unique correspondence in $\MM(\Lambda^k)$) and $c>0$,  such that $\mu_s = \Pi_1^{-1}(\phi^{-1}\circ\nu_k)\in\MM_s(\Lambda^k)$ and for the $(Z_m)_{m=1}^k$ defined as above using $\mu_s$,
\begin{equation}
    W:=
    \log\bigg[
        \frac{k-2}{k-1}+
        \frac{1}{\sum_{m=2}^{k}\exp(Z_{1}-Z_{m})}
    \bigg] \vee 0
    \succ_{c/logk} \nu_k
    ,
\label{eq:col_rec_W}
\end{equation}
then $\mu_s$ satisfies the requirement of Theorem~\ref{thm:col_rec_stocdom}.
\end{prop}
\begin{proof}
    Maximizing \eqref{eq:col_rec_treerec3} over $m\in[k]$, we have that
\begin{align*}
\|\vX_{n+1}\|_{\infty}
=\frac{\max\{1,\exp(Z_{1}-Z_{m}),m=2,\dots,k\}}
    {1+\sum_{m=2}^{k}\exp(Z_{1}-Z_{m})}
\ge
\frac{1}{1+\sum_{m=2}^k \exp(Z_1 - Z_m)}
    \vee\frac{1}{k}.
\end{align*}
Composing $\phi$ to both side yields that $\phi(\|\vX_{n+1}\|_\infty)\succ W$.
Theorem~\ref{thm:col_rec_stocdom} then follows from the fact that $\|\Lambda(\vX_{n+1})\|_{\infty}=\|\vX_{n+1}\|_{\infty}$.
\end{proof}

We now propose a parameterized candidate of $\nu_k$:
Let $\delta,\kappa\in(0,1),M\gg 0,0<\gamma,\alpha_{0},\sigma,\epsilon\ll 1$
be parameters to be determined in the order of $(\delta,\kappa,\alpha_{0},M,\sigma,\gamma,\epsilon)$
and write $\alpha=\phi\big(\frac{1}{2}-\alpha_{0}\big)=\log2-O(\alpha_{0})+o_{k}(1)$.
Let $\nu_\star$ be an infinite-volume measure defined as (recalling that $\phi(\frac{1}{k})=0$)
\begin{equation}\label{eq:nu_star}
\nu_\star(dy):=\kappa\delta_{0}(dy)+(1-\kappa)\delta_{\alpha}(dy)+\frac{\gamma}{y^{2}}e^{\delta y}\Ind{y>M}dy,
\end{equation}
where $\delta_{x}$ is the Dirac measure at $x$, and write $\nu_{r}(dy):=\frac{\gamma}{y^{2}}e^{\delta y}\Ind{y>M}dy$
for the right tail of $\nu_\star$. We will use $\nu_\star$ as a ``scaling limit'' of $\nu_{k}$ and show that the assumption of Prop.~\ref{prop:col_rec_rewrite} is satisfied with
\[
\nu_{k}(dy):=\frac{1}{\log k}\nu_\star(dy)\Ind{0\le y\le a_{k}},
\]
for some choice of $(\delta,\kappa,\alpha_{0},M,\sigma,\gamma,\epsilon)$ and $k\ge k_0 = k_0(\delta,\kappa,\alpha_{0},M,\sigma,\gamma,\epsilon)$, where $a_{k}$ is the constant such that $\nu_{k}$ is a probability measure.

For convenience of notation, we will write $k\ge k_0$ where $k_0$ depends on all six parameters. We will use $1_{\le a_{k}}$ or $1_{\ge c_{k}}$ to cut (part of) a measure above or below such that the total mass is 1. The exact value of $a_{k}$ and $c_{k}$ can be derived implicitly and may vary from line to line.
 Let $\nu_\star^{=}(dy):=\phi\circ\mu^{=}(dx)=\phi\circ x\mu(dx)=\phi^{-1}(y)\nu_\star(dy)$, where $\phi^{-1}(y)=1-\left(e^{y}+(k-1)^{-1}\right)^{-1}$ and define
$\nu_\star^{\neq},\nu_{r}^{=},\nu_{r}^{\neq},\nu_{k}^{=},\nu_{k}^{\neq}$
similarly. We define the tail weights
\begin{align*}
p_{r}^{\neq} & :=\frac{1}{\gamma}\nu_{r}^{\neq}([M,\infty))=\int_{M}^{\infty}\frac{e^{\delta y}}{y^{2}(e^{y}+\frac{1}{k-1})}dy<\infty\\
p_{k}^{\neq} & :=
\mu_{k}^{\neq}([1/k,1))=
\nu_{k}^{\neq}([0,\infty))
\\
& \le\frac{1}{\log k}\left[\frac{1}{k}(1-\kappa)+\Big(\frac{1}{2}-\alpha_{0}\Big)\kappa+\gamma p_{r}^{\neq}\right]
= (1-o(1))\frac{\gamma p^{\neq}_r}{\log k}.
\end{align*}

\subsection{Distribution of \texorpdfstring{$Z_m$}{Zm}}
In this section we bound the distribution of $Z_m$ in terms of $\nu_\star$.
Let $D:=d/(k-1)=\log k+\log\log k+\beta$. For $T\sim\mathcal{T}_{\pois(d)}$,
(\ref{eq:col_rec_cmrate}) implies that $(c_{m}^{=},c_{m}^{\neq})$'s are independent Poisson random variables with rate $(0,p_{k}^{\neq}D)$ for $m=1$ and $((1-p_{k}^{\neq})D,\frac{k-2}{k-1}p_{k}^{\neq}D)$ for $m\ge2$.
Hence, for $m\ge2$,
\begin{align*}
Z_{m} &\deq
\bigg(
    \pois((1-p_{k}^{\neq})D)
    \otimes\frac{1}{1-p_{k}^{\neq}}\nu_{k}^{=}
\bigg)
\oplus \bigg(
    \pois\Big(\frac{k-2}{k-1}p_{k}^{\neq}D\Big)
    \otimes\frac{1}{p_{k}^{\neq}}\nu_{k}^{\neq}
\bigg)
\\ & =\pois\bigg(\Big(1-\frac{p_{k}^{\neq}}{k-1}\Big)D\bigg)\otimes\frac{\nu_{k}^{=}+\frac{k-2}{k-1}\nu_{k}^{\neq}}{(1-\frac{1}{k-1}p_{k}^{\neq})}
\succ
\pois\bigg(\Big(1-\frac{p_{k}^{\neq}}{k-1}\Big)D\bigg)\otimes\frac{\nu_{k}}{(1-\frac{1}{k-1}p_{k}^{\neq})}1_{\le a_{k}},
\end{align*}
where the last line follows from that $(\nu_{k}^{=}+\frac{k-2}{k-1}\nu_{k}^{\neq})(dy)\le \nu_{k}(dy)$.
Namely, $Z_m$ stochastically dominates the sum of points in a Poisson point process with intensity $D\nu_k1_{\le a^0_k}$, where $a_k^0$ satisfies $\nu_k([0,a_k^0])=1-\frac{1}{k-1}p_{k}^{\neq}$. We expand the summation according to the three parts of $\nu_k$ as in \eqref{eq:nu_star}. Firstly, $\delta_0$ does not contribute to the summation. For the second term, we define
$ S_1:=\pois(\kappa)\otimes \delta_\alpha $ and note that $\kappa\le\frac{\kappa D}{\log k}$. Finally for $k\ge k_0$, the total intensity coming from the right tail of $\nu_k$ satisfies
\[
    D\nu_k([M,a^0_k])
    = D(\nu_k([0,a^0_k])-\log^{-1}k)
    = D - 1 - O(k^{-1}\log k) \ge D-1-\gamma
    .
\]
and $(1-\frac{1}{k-1}p_{k}^{\neq})^{-1}\nu_k(dy)\le \frac{1+\gamma}{D-1-\gamma} \nu_r(dy)$. Therefore defining
\[
    S_0 := \pois\left(D-1-\gamma\right)\otimes\frac{1+\gamma}{D-1-\gamma}\nu_{r}1_{\le a_{k}},
\]
it follows that $Z_m\succ S_0+ S_1$.  We first show the following bound for $S_{0}$.
\begin{lem}
\label{lem:col_zm_S0bound}For any $M>M(\alpha_{0})\vee\frac{2}{\delta}$,
there exists constant $C_{M}>0$ such that
\begin{align}
S_{0} & \succ\frac{e^{\gamma+1-\beta}}{k\log k}(\delta_{0}+(1+C_{M}\gamma)\nu_{r}1_{\le a^1_{k}}),\label{eq:col_zm_S0bound}
\end{align}
where $a^1_k$ satisfies $1+(1+C_M\gamma)\nu_r([M,a^1_k])= k\log k e^{-(\gamma+1-\beta)}$.
\end{lem}
\begin{proof}
Let $B_{0}\sim\pois\left(D-1-\gamma\right)$ and $Y_{i}$ be i.i.d.\ samples of distribution $\frac{1+\gamma}{D-1-\gamma}\nu_{r}1_{\le a_{k}}$. We have
\[
    P(S_{0}=0) = P(B_{0}=0) = e^{-(D-1-\gamma)}
    \le\frac{1}{k\log k}e^{1+\gamma-\beta}
    .
\]
Since $\nu_{r}$ is supported on $[M,\infty)$ and is absolutely continuous,
for $z\ge M$,
\begin{align*}
f_{S_{0}}(z)
& =\frac{d}{dz}P\Big(\sum_{i=1}^{B_{0}}Y_{i}\le z\Big)
\le \sum_{n=1}^{\lfloor z/M\rfloor}
P(B_{0}=n)
    \frac{d}{dz}\bigg[
    \int_{\sum y_{i}\le z}
    \left(\frac{1+\gamma}{D-1-\gamma}\right)^{n}
    \nu_{r}(dy_{1})\cdots\nu_{r}(dy_{n})
   \bigg]
\\&\le \frac{e^{1+\gamma-\beta}}{k\log k}
   \sum_{n=1}^{\lfloor z/M\rfloor}
   \frac{1}{n!}
   \frac{d}{dz}\bigg[
        \int_{y_{i}\ge M,\sum_{i=1}^{n}y_{i}\le z}
        \frac{(1+\gamma)^{n}\gamma^{n}}
        {y_{1}^{2}y_{2}^{2}\cdots y_{n}^{2}}
        e^{\delta(y_{1}+\cdots+y_{n})}
        dy_{1}\cdots dy_{n}
    \bigg]
\\& =\frac{e^{1+\gamma-\beta}}{k\log k}\sum_{n=1}^{\lfloor z/M\rfloor}\frac{(1+\gamma)^{n}\gamma^{n}}{n!}e^{\delta z}\int_{y_{i}\ge M,\sum_{i=1}^{n-1}y_{i}\le z-M}\frac{1}{y_{1}^{2}\cdots y_{n-1}^{2}(z-\sum_{i=1}^{n-1}y_{i})^{2}}dy_{1}\cdots dy_{n-1}.
\end{align*}
Applying Fact~\ref{fact:col_zm_int_x} below for $n\ge2$, we have that
for $z\ge M$,
\begin{align*}
f_{S_{0}}(z)dz & \le\frac{e^{1+\gamma-\beta}}{k\log k}
\bigg((1+\gamma)\gamma+\sum_{n=2}^{\infty}\frac{((1+\gamma)\gamma C_{M})^{n}}{n!}\bigg)\frac{1}{z^{2}}e^{\delta z}dz
\\&\le \frac{e^{1+\gamma-\beta}}{k\log k}(1+C_{M}'\gamma)\frac{\gamma}{z^{2}}e^{\delta z}dz=\frac{e^{1+\gamma-\beta}}{k\log k}(1+C'_{M}\gamma)\nu_{r}(dz).
\end{align*}
The desired result follows from the last equation and the fact that
$P(S_{0}\in(0,M))=0$.
\end{proof}
\begin{fact}
\label{fact:col_zm_int_x}There exist constant $C_{M}$ such that
for $n\ge2$ and $z\ge nM$,
\[
\int_{y_{i}\ge M,\sum_{i=1}^{n-1}y_{i}\le z-M}\frac{1}{y_{1}^{2}\cdots y_{n-1}^{2}(z-\sum_{i=1}^{n-1}y_{i})^{2}}dy_{1}\cdots dy_{n-1}\le\frac{C_{M}^{n}}{z^{2}}.
\]
\end{fact}

The proof of Fact~\ref{fact:col_zm_int_x} is postponed to Section~\ref{sec:appd}.  Next consider the independent sum of $S_{0}+S_{1}$.
\begin{lem}
\label{lem:col_zm_S0S1bound}For any $M>M(\alpha_{0})\vee\frac{2}{\delta}$
and constant $C_{M}$ specified in Lemma~\ref{lem:col_zm_S0bound},
\begin{equation}
Z_{m}\succ S_{0}+S_{1}\succ\frac{e^{\gamma+1-\beta}}{k\log k}\big[\nu_{S_{1}}+(1+\alpha_{0})(1+C_{M}\gamma)\exp\big(\kappa(e^{-\alpha\delta}-1)\big)\nu_{r}1_{\le a_{k}}\big].\label{eq:col_zm_zmbound}
\end{equation}
\end{lem}
\begin{proof}
Letting $\nu_{S_{0}}^{+}:=\nu_{r}1_{\le a^1_{k}}$ where $a^1_{k}$ is defined in (\ref{eq:col_zm_S0bound}), we have
\begin{equation}
\nu_{S_{0}+S_{1}}=\frac{e^{1+\gamma-\beta}}{k\log k}(\delta_{0}*\nu_{S_{1}}+(1+C_{M}\gamma)\nu_{S_{0}}^{+}*\nu_{S_{1}})=\frac{e^{1+\gamma-\beta}}{k\log k}(\nu_{S_{1}}+(1+C_{M}\gamma)\nu_{S_{0}}^{+}*\nu_{S_{1}}).\label{eq:col_zm_S0S1}
\end{equation}
It is left to verify that $\nu_{S_{0}+S_{1}}^{+}:=\nu_{S_{0}}^{+}*\nu_{S_{1}}\succ(1+\alpha_{0})\exp\big(\kappa(e^{-\alpha\delta}-1)\big)\nu_{r}1_{\le a_{k}}$
where $a_{k}$ is chosen such that RHS of (\ref{eq:col_zm_S0S1}) has total mass 1. Recall that $S_{1}\deq\alpha\cdot\pois(\kappa)$. $\nu_{S_{0}+S_{1}}^{+}$ is absolutely continuous and supported on $[M,\infty)$.
For $z\ge M$ we have
\begin{align*}
f_{S_{0}+S_{1}}^{+}(z)
&=\sum_{n=0}^{\infty}
    \frac{\kappa^{n}e^{-\kappa}}{n!}
    f_{S_{0}}^{+}(z-n\alpha)
\le\sum_{n=0}^{\infty}
\frac{\kappa^{n}e^{-\kappa}}{n!}
\frac{\gamma e^{\delta(z-n\alpha)}}{(z-n\alpha)^{2}}
\Ind{z-n\alpha\ge M}\\
\end{align*}
To control the $(z-n\alpha)^{-2}$ term, we first choose for any $\alpha>0$ a $N=N(\alpha_0)$ such that $\sum_{n=N+1}^{\infty}\frac{1}{n!}\le\frac{1}{2e}\alpha_{0}$ and then choose $M(\alpha_{0})$ such that for $M>M(\alpha_{0})$, $n\le N$ and $z\ge M$,
\begin{equation}
    (1-n\alpha/z)^{-2}\le (1-n\alpha/z)^{-2} \le 1 + \alpha_{0}/2
.  \label{eq:col_zm_Mconstraint1}
\end{equation}
Observe that $\frac{\gamma}{z^{2}}e^{\delta z}$ is monotone increasing for $z\in(\frac{2}{\delta},\infty)$. For all $M>M(\alpha_{0})\vee\frac{2}{\delta}$ and $z\ge M$,
\begin{align*}
f_{S_{0}+S_{1}}^{+}(z)dz
&\le \frac{\gamma e^{\delta z}}{z^{2}}dz
    \sum_{n=0}^{N} \frac{\kappa^{n}e^{-\kappa}}{n!}
    \frac{e^{-n(\alpha\delta)}}{(1-n\alpha/z)^{2}}
+ \frac{\gamma e^{\delta z}}{z^{2}}dz\sum_{n=N+1}^{\infty}\frac{\kappa^{n}e^{-\kappa}}{n!}\\
 & \le(1+\alpha_{0})\exp[\kappa(e{}^{-\alpha\delta}-1)]\nu_{r}(dz).
\end{align*}
The proof finishes by cutting $\nu_{r}$ at the place such that (\ref{eq:col_zm_S0S1}) has the total mass 1.
\end{proof}
Finally, for $m=1$ and $k\ge k_0$ such that $\frac{D}{\log k}\le(1+\gamma)\vee(1+\alpha_{0})$,
we have
\begin{equation}
Z_{1}
\deq \pois(p_{k}^{\neq}D)
    \otimes \frac{1}{p_{k}^{\neq}}\nu_{k}^{\neq}
\prec
\left(
    \pois\left(\frac{1}{2}\kappa\right)\otimes\delta_{\alpha}
\right)\oplus\left(
    \pois(\gamma p_{r}^{\neq})
    \otimes\frac{1}{\gamma p_{r}^{\neq}}\nu_{r}^{\neq}
\right)
,\label{eq:col_zm_z1}
\end{equation}
where the second term is $0$ with probability $\exp(-\gamma p_{r}^{\neq})$.

\subsection{Distribution of \texorpdfstring{$\sum_{m=2}^{k}\exp(-Z_{m})$}{the sum of exp(-Zm)}}

In this section we analysis the distribution of $\sum_{m=2}^{k}\exp(-Z_{m})= (k-1)\otimes \exp(-Z_m)$.  Let $\psi(x):=e^{-x}$. An easy calculation gives that
\[
    \psi\circ\nu_\star(dx)
    = (1-\kappa)\delta_{1}(dx)
    + \kappa\delta_{\psi(\alpha)}(dx)
    + \frac{\gamma}{(\log x)^{2}}\frac{1}{x^{1+\delta}}
        \Ind{0<x<\psi(M)}dx.
\]
Define
\begin{equation}\label{eq:col_sumzm_CZ}
    C_Z:=C_Z(\delta,\kappa,\alpha_0,M,\gamma)
    =(1+\alpha_{0})(1+C_{M}\gamma)\exp\big(\kappa(e^{-\alpha\delta}-1)\big).
\end{equation}
Now (\ref{eq:col_zm_zmbound}) can be rewritten  as
\begin{equation}\label{eq:col_sumZm_psiZm}
    \psi(Z_{m}) \prec
    \frac{1}{k\log k}e^{\gamma+1-\beta}
    \Big[
        \psi\circ\nu_{S_{1}}
        + C_{Z}\frac{\gamma}{(\log x)^{2}}\frac{1}{x^{1+\delta}}
            \Ind{c_{k}<x<\psi(M)}
    \Big]
    .
\end{equation}

As $k$ grows, the density of $\psi(Z_{m})$ diverges quickly around $0$ and the probability of seeing $Z_m\ge x$ for more than one $m\in[k]$ is $o(\frac{1}{k})$ for any fixed $x>0$. Hence intuitively,
\[
\nu_{k\otimes\psi(Z_{m})}\approx\nu_{\max_{m\in[k]}\psi(Z_{m})}\approx k\cdot\nu_{\psi(Z_{m})}.
\]
\begin{lem}
\label{lem:col_sumzm}Fix $\delta=\frac{1}{2}$. For any $M>M(\alpha_{0})\vee\frac{2}{\delta}$ such that (\ref{eq:col_zm_S0S1}) holds and $\sigma,\epsilon>0$, $k\ge k_0$,
\begin{align*}
(k-1)\otimes\psi(Z_{m})
&\prec
\frac{e^{\gamma+1-\beta}}{\log k}
\left[
    (\psi+\sigma)\circ\nu_{S_{1}}
    + (1+\epsilon)C_{Z}
        \frac{\gamma}{(\log x)^{2}}\frac{1}{x^{1+\delta}}
        1_{x\le\psi(M)}
\right]1_{\ge c_{k}}
+ \frac{\epsilon}{\log k}\delta_{\infty},
\end{align*}
where $(\psi+\sigma)(x):=\psi(x)+\sigma$ and $C_Z$ is defined in \eqref{eq:col_sumzm_CZ}.
\end{lem}
\begin{proof}
We recall the RHS of \eqref{eq:col_sumZm_psiZm} and treat its discrete part and continuous part separately. Let $p_{1}:=\frac{e^{\gamma+1-\beta}}{k\log k}$, $\mu_{Z}^{1}:=\psi\circ\nu_{S_{1}}$ and $\mu_{Z}^{2}(dx):=\frac{p_{1}}{1-p_{1}}\frac{\gamma}{(\log x)^{2}}x^{-(1+\delta)}1_{c_{k}<x\le\psi(M)}dx$.
Among the $(k-1)$ i.i.d.\ samples from the RHS of \eqref{eq:col_sumZm_psiZm},
$b\sim\mathrm{Binom}(k-1,p_{1})$ of them comes from $\mu_{Z}^{1}$ and the rest comes from $\mu_{Z}^{2}$. Choose $C_{b}>0$ such that for any $k\ge k_0$, $P(b\ge2)\le C_{b}\log^{-2}k$. It follows that
\begin{align}
(k-1)\otimes\psi(Z_{m}) & \prec\left(\mathrm{Binom}(k,p_{1})\otimes\mu_{Z}^{1}\right)\oplus\left(k\otimes\mu_{Z}^{2}\right)\nonumber \\
&\prec
\Big[
    (1-kp_{1})\cdot k\otimes\mu_{Z}^{2}
    + kp_{1}\cdot\left(\mu_{Z}^{1}\oplus(k\otimes\mu_{Z}^{2})\right)
\Big]1_{\ge c_k}
    +\frac{C_{b}}{\log^{2}k}\delta_{\infty}
    .\label{eq:col_sumZm_temp2}
\end{align}
We will show in Lemma~\ref{lem:col_sumzm_ktimesU} that for any $\epsilon>0$ and $k\ge k_0$,
\begin{equation}
k\otimes\mu_{Z}^{2}
\prec(1+\epsilon)k\cdot\mu_{Z}^{2}1_{\ge c^0_{k}}
+\frac{\epsilon}{2\log k}\delta_{\infty}.\label{eq:col_sumZm_contpart}
\end{equation}
Therefore for any $\sigma>0$, there exists $C_{\sigma}>0$ such that for $k\ge k_0$, $P(k\otimes\mu_{Z}^{2}\ge\sigma)\le C_{\sigma}\log^{-1}k$ and
\begin{align*}
\textup{RHS of \eqref{eq:col_sumZm_temp2}}
&\prec
\Big[
(1-kp_{1})\cdot k\otimes\mu_{Z}^{2}
+ kp_{1}\cdot(\mu_{Z}^{1}*\delta_{\sigma})
+ \frac{kp_{1}C_\sigma}{\log k}\cdot\delta_{\infty}
\Big]1_{\ge c_k}
+\frac{C_{b}}{\log^{2}k}\delta_{\infty}
 \\&\prec
\Big[
(1+\epsilon)k(1-kp_{1})\cdot\mu_{Z}^{2}1_{\ge c^0_k}
+ kp_{1}\cdot(\mu_{Z}^{1}*\delta_{\sigma})
\Big]1_{\ge c_k}
+\frac{\epsilon}{\log k}\delta_{\infty}
\\&\prec
\frac{e^{\gamma+1-\beta}}{\log k}
\left[(\psi+\sigma)\circ\nu_{S_{1}}
    +(1+\epsilon)C_{Z}\frac{\gamma}{(\log x)^{2}}\frac{1}{x^{1+\delta}}
\right] 1_{\ge c_{k}}
+\frac{\epsilon}{\log k}\delta_{\infty}
.
\end{align*}
where in the last step, we observe that removing the $1_{\ge c^0_k}$ after $\mu^2_Z$ will only make the measure inside the square bracket stochastically larger after cutting from below.
\end{proof}
In the remaining of the section, we check that \eqref{eq:col_sumZm_contpart} is true. We will
henceforth omit the $O(1)$ factor $(k\log k)\cdot\frac{p_{1}}{1-p_{1}}$
by absorbing it into $\gamma$ and let
\begin{equation}\label{eq:col_sumZm_defU}
U\sim\mu_{U} :=\mu_{Z}^{2}
=\frac{1}{k\log k}\frac{\gamma}{(\log x)^{2}}
x^{-(1+\delta)}\Ind{c_{k}<x\le\psi(M)}dx.
\end{equation}
Measure $\mu_U$ resembles distributions that converge to stable law. However, we can not directly apply the usual proof of convergence for stable laws (cf.\ Section 3.7 of \cite{durrett2010probability}, or the reference there) to $k\otimes U$, since the expression of $\mu_{U}$ also depends on $k.$ With some modification, we show the following result.
\begin{lem}
\label{lem:col_sumzm_stablelaw}For any $\delta,\gamma\in(0,1),$
$M>\frac{2}{\delta}$, let $t_{k}:=\inf\{t:\mu_{U}([t,\infty))<1/k\}$,
then $k\otimes(t_{k}^{-1}U)$ converges weakly to the stable law with
index $\delta$ and characteristic function
\[
    \exp\{-b_\star|t|^{\delta}(1+i\mathrm{sgn}(t)\tan(\pi\delta/2))\},
\]
where $\mathrm{sgn}$ is the sign function and $b_\star=\delta\int_{0}^{\infty}(\cos x-1)x^{-(1+\delta)}dx=-\cos(\frac{\pi}{2}\delta)\Gamma(1-\delta)$.
\end{lem}
In the proof we use the following calculus result, the proof of which is deferred to Section \ref{sec:appd}.
\begin{fact}
\label{fact:col_sumzm_stableintegr}Let $t_{k}$ be defined as in
Lemma \ref{lem:col_sumzm_stablelaw}, we have
\begin{enumerate}
\item $t_{k}=(1+o_{k}(1))(\frac{\gamma\delta}{\log k(\log\log k)^{2}})^{1/\delta}$
and therefore
\[
\frac{\gamma}{\delta}t_{k}^{-\delta}\log^{-2}t_{k}=(1+o_{k}(1))\log k.
\]

\item For any constant $c>0$,
\begin{align*}
\lim_{k\to\infty}kP(U\ge ct_{k}) & =\lim_{k\to\infty}t_{k}^{-1}\int_{ct_{k}}^{\infty}\frac{1}{k\log k}\frac{\gamma}{\log^{2}x}\frac{1}{x^{1+\delta}}dx=c^{-\delta},\\
\lim_{k\to\infty}k\E(t_{k}^{-1}U1_{U\le ct_{k}}) & =\lim_{k\to\infty}t_{k}^{-1}k\int_{0}^{ct_{k}}\frac{1}{k\log k}\frac{\gamma}{\log^{2}x}\frac{x}{x^{1+\delta}}dx=c^{1-\delta}\frac{\delta}{1-\delta},\\
\lim_{k\to\infty}k\E(t_{k}^{-2}U^{2}1_{U\le ct_{k}}) & =\lim_{k\to\infty}t_{k}^{-2}k\int_{0}^{ct_{k}}\frac{1}{k\log k}\frac{\gamma}{\log^{2}x}\frac{x^{2}}{x^{1+\delta}}dx=c^{2-\delta}\frac{\delta}{2-\delta}.
\end{align*}

\end{enumerate}
\end{fact}
\begin{proof}[Proof of Lemma~\ref{lem:col_sumzm_stablelaw}]
Let $U_{i},i=1,2,\dots,k$ be i.i.d.\ copies of $U$ and let $S_{k}:=\sum_{i=1}^{k}U_{i}$.
Given $\omega\in(0,1)$, let $m_{\le\omega}:=\E(U1\{U\le\omega t_{k}\})$,
$S_{k}^{\omega}:=\sum_{i=1}^{k}U_{i}\Ind{U_{i}\ge\omega t_{k}}$
and $T_{k}^{\omega}:=\sum_{i=1}^{k}U_{i} \Ind{U_{i}<\omega t_{k}}-km_{\le\omega}$.
We have
\[
S_{k}=S_{k}^{\omega}+T_{k}^{\omega}+k\cdot m_{\le\omega}.
\]
For the first term $S_{k}^{\omega}$, let $F_{k}^{\omega}$ and
$\psi_{k}^{\omega}$ be the c.d.f.\ and characteristic function
of $t_{k}^{-1}U_{i}$ conditioned on $\{t_{k}^{-1}U_{i}\ge\omega\}$.
By Fact~\ref{fact:col_sumzm_stableintegr}(2), for any $\omega>0$ and
any $x>\omega$,
\[
1-F_{k}^{\omega}(x)
=(1+o_{k}(1))(x/\omega)^{-\delta}
\to(\omega/x)^{\delta}
,\quad\textup{as }k\to\infty.
\]
Hence for any $t\in\mathbb{R}$, $\psi_{k}^{\omega}(t)\to\psi^{\omega}(t):=\int_{\omega}^{\infty}e^{itx}\cdot\delta\omega^{\delta}x^{-(\delta+1)}dx$.
Meanwhile by Fact~\ref{fact:col_sumzm_stableintegr}(2), the distribution
of the number of $i\in[k]$ such that $U_{i}\ge \omega t_{k}$ converges
weakly to $\pois(\omega^{-\delta})$, hence
\[
\lim_{k\to\infty}\E\exp(itS_{k}^{\omega}/t_{k})=\exp[-\omega^{-\delta}(1-\psi^{\omega}(t))]=\exp\left(\int_{\omega}^{\infty}(e^{itx}-1)\delta x^{-(\delta+1)}dx\right).
\]
For the second term $T_{k}^{\omega}$, observe that $\E T_k^\omega = 0$. By Fact~\ref{fact:col_sumzm_stableintegr},
\begin{align*}
t_{k}^{-2}\E(T_{k}^{\omega})^{2} & =t_{k}^{-2}\Var(T_{k}^{\omega})\le kt_{k}^{-2}\E U_{i}^{2}1\{U_{i}<\omega t_{k}\}\le(1+o_{k}(1))\frac{\delta}{2-\delta}\omega^{2-\delta}.
\end{align*}
For each $t\in\mathbb{R}$, $\exp(itx)$ is a Lipschitz function
with Lipschitz constant $t$. By Jensen's inequality,
\[
|\E\exp(it(t_{k}^{-1}S_{k}))-\E\exp(it(t_{k}^{-1}S_{k}^{\omega}))|\le t\left(\E|t_{k}^{-1}T_{k}^{\omega}|+t_{k}^{-1}km_{\le\omega}\right)\le O(\omega^{1-\delta/2}).
\]
Let $\omega\to0$. By dominated convergence theorem, we have
\[
\lim_{k\to\infty}\E(\exp(itS_{k}/t_{k}))=\exp\left(\int_{0}^{\infty}(e^{itx}-1)\delta x^{-(\delta+1)}dx\right).
\]
The rest of the proof follows from complex analysis: Let $\Gamma$
denote the gamma function (not to be confused with the recursion $\Gamma_{s}$
defined before). For $t>0$, (the case of $t<0$ is parallel)
\begin{align*}
\int_{0}^{\infty}(e^{itx}-1)\delta x^{-(\delta+1)}dx & =t^{\delta}\int_{0}^{\infty}(e^{ix}-1)\delta x^{-(1+\delta)}dx\\
 & =it^{\delta}\int_{0}^{\infty}x^{-\delta}e^{ix}dx=i^{\delta}t^{\delta}\int_{0}^{\infty}(ix)^{-\delta}e^{ix}d(ix)\\
 & =\Gamma(1-\delta)i^{\delta}t^{\delta}=\cos(\pi\delta/2)\Gamma(1-\delta)t^{\delta}(1+i\tan(\pi\delta/2)),
\end{align*}
where the second equality follows by integration by part and the last
equality follows by doing contour integral on region $\{re^{i\theta}:\omega\le r\le R,\theta\in[0,\frac{\pi}{2}]\}$
and letting $\omega\to0,R\to\infty$.
\end{proof}
Let $\widetilde{U}$ denote the limiting stable law specified in Lemma~\ref{lem:col_sumzm_stablelaw}.  When $\delta=\frac{1}{2}$, $\widetilde{U}$ follows the Levy distribution with parameter $\frac{\pi}{2}$. Since this is the only value of $\delta$ for which we have a closed formula for $f_{\widetilde{U}}$, here and henceforth we will take $\delta=1/2$.  The result, however, should hold for all $\delta\le\frac{1}{2}$ as long as \eqref{eq:col_sumzm_generaldelta} holds.
Plugging in the formula of Levy distribution and comparing with Fact~\ref{fact:col_sumzm_stableintegr}, we have
\begin{align}
P(\widetilde{U}\le c)
&=\frac{2}{\sqrt{\pi}}\int_{\frac{1}{2}\sqrt{\pi/c}}^{\infty}e^{-t^{2}}dt\le\frac{2}{\sqrt{\pi}}\frac{1}{2}\sqrt{\frac{\pi}{c}}e^{-\pi/2c}\le c^{-1/2}e^{-\pi/2c}\nonumber \\
 & \stackrel{}{<}c^{-1/2}=(1+o_{k}(1))kP(U<ct_{k})
.\label{eq:col_sumzm_generaldelta}
\end{align}
Thus we can upper-bound $\mu_{k\otimes U}(dx)$ by $(1+o_{k}(1))k\cdot\mu_{U}(dx)$ for small $x\approx O(t_{k})$. In the next lemma, we bound larger values of $k\otimes U$ using the intuition of $k\otimes U\approx\max_{i=1,\dots,k}U_{i}$.
\begin{lem} \label{lem:col_sumzm_ktimesU}
Fix $\delta=1/2$. For any $M\ge\frac{2}{\delta}$,
$\gamma,\epsilon\in(0,1)$, and $k\ge k_0$,
\begin{equation}
k\otimes\mu_{U}\prec(1+\epsilon)k\cdot\mu_{U}1_{\ge c_{k}}+\frac{\epsilon}{\log k}\delta_{\infty}.\label{eq:col_sumzm_ktimesU}
\end{equation}
\end{lem}
\begin{proof}
Let $U_{1},\dots,U_{k}$ be i.i.d.\ copies of $U$ and define $U_{(1)}:=\max_{i=1,\dots,k}U_{i}$, $U_{R}:=\sum_{i=1}^{k}U_{i}-U_{(1)}$. Let $c=c(\delta,M,\gamma,\epsilon)>0$ be some small constant to be determined. We write
\begin{equation}
P\Big(\sum_{i=1}^{k}U_{i}\ge z\Big)
\le P(U_{(1)}\ge(1-c)z)
+ \int_{0}^{(1-c)z}f_{U_{(1)}}(x)P(U_{R}\ge z-x\mid U_{(1)}=x)dx
,\label{eq:col_kU_sumU}
\end{equation}
where $f_{U_{(1)}}(z)=kf_{U}(z)(F_{U}(z))^{k-1}\le kf_{U}(z)$. Fix $\sigma=\sigma(\delta,M,\gamma,\epsilon)\in(0,\frac{1}{2})$
such that
\[
P(U\ge(1-\sigma)\psi(M))\le\frac{1}{\log k}\int_{(1-\sigma)\psi(M)}^{\psi(M)}\frac{\gamma}{\log^{2}x}x^{-(1+\delta)}dx\le\frac{\epsilon}{2\log k}.
\]
 We will split the proof into three cases: $x\in[c_{k},Nt_{k}]$,
 $x\in[Nt_{k},(1-\sigma)\psi(M)]$ and $x\ge(1-\sigma)\psi(M)$ where $N=N(\delta,M,\gamma,\epsilon,\sigma,c)$ is a large constant to be determined.
\begin{enumerate}
\item $x\in[Nt_{k},(1-\sigma)\psi(M)]$:
To bound the first term of \eqref{eq:col_kU_sumU}, we observe that $f_{U}$ is a decreasing function and for $z\le (1-\sigma)\psi(M)$, $(1+\sigma)z \le \psi(M)\in\supp U$. Therefore
\[
    \frac{P(U_{(1)}\in [(1-c)z,z])}
    {P(U_{(1)}\in [z,(1+\sigma)z])}
    \le \frac{czf_U((1-c)z)F^{k-1}(z)}
        {\sigma zf_U((1+\sigma)z)F^{k-1}(z)}
    \le \frac{c}{\sigma}
    \frac{f_U(z/2)} {f_U((1+\sigma)z)}
    \le C_{\sigma,M}\cdot c,
\]
for all $c\le 1/2$ and $z\le (1-\sigma)\psi(M)$. It follows that
\begin{equation} \label{eq:col_kU_sumU_part1}
P(U_{(1)}\ge (1-c)z)
\le (1+C_{\sigma,M}\cdot c) P(U_{(1)}\ge z)
\le (1+C_{\sigma,M}\cdot c) kP(U\ge z)
.
\end{equation}
For the second term of \eqref{eq:col_kU_sumU}, a similar calculation of Fact~\ref{fact:col_sumzm_stableintegr} gives that for any $x\le\psi(M)$,
\begin{align*}
k\log k\E(U\mid U\le x)
&=\frac{k\log k}{F_{U}(x)}
\int_{0}^{x}zf_{U}(z)dz
\le\frac{\gamma}{1-\delta}\frac{1}{F_{U}(x)}\frac{1}{\log^{2}x}x^{1-\delta},\\
k\log k\E(U^{2}\mid U\le x)
& =\frac{k\log k}{F_{U}(x)}\int_{0}^{x}z^{2}f_{U}(z)dz
\le \frac{\gamma}{2-\delta} \frac{1}{F_{U}(x)}
    \frac{1}{\log^{2}x} x^{2-\delta}.
\end{align*}
Recall the expression of $t_k$ from Fact~\ref{fact:col_sumzm_stableintegr}. For any $c>0$ we choose
$N=N(M,\gamma,\epsilon,c)$ such that for $k\ge k_0$ and $x\ge Nt_{k}$,
\begin{equation}
k\E(U\mid U\le x)
\le
\frac{1+o_k(1)}{\log k}
\frac{\gamma}{1-\delta}
\frac{x(Nt_{k})^{-\delta}}{\log^2 t_{k}}
=(1+o_{k}(1))N^{-\delta}
\frac{\delta}{1-\delta}
x\le\frac{1}{2}cx.\label{eq:col_kU_Nconstraint1}
\end{equation}
Given $U_{(1)} =x$, $U_R$ is distributed as the sum of $(k-1)$ i.i.d.\ copies of $U$ conditioned on $U\le x$. By Chebyshev inequality, for any $z\in[2Nt_{k},\psi(M)]$ and $x\le (1-c)z$,
\[
    P(U_{R}\ge z-x\mid U_{(1)}=x)
    \le
    \frac{k\cdot\E(U^2\mid U\le x)}{(z-x-k\E(U\mid U\le x))^{2}}
    \le
    \frac{4}{c^2z^2}
    \frac{1}{\log k}
    \frac{\gamma}{2-\delta}\frac{1}{F_{U}(x)}
    \frac{x^{2-\delta}}{\log^{2}x}
    ,
\]
where in the second step, we use the fact that $\E(U\mid U\le x)$ is monotone decreasing in $x$. Plugging the estimation into the RHS of \eqref{eq:col_kU_sumU}, for $z\le \psi(M)$, we have that
\begin{align}
\int_{0}^{(1-c)z}kf_{U}(x)F_U(x)^{k-1}
&P(U_{R}\ge z-x\mid U_{(1)}=x)dx
\nonumber\\& \le\int_{c_k}^{(1-c)z}
    \frac{1}{\log k} \frac{\gamma}{\log^2 x} x^{-1-\delta}
    \cdot
    \frac{4}{c^2z^2}
    \frac{1}{\log k}
    \frac{\gamma}{2-\delta}
    \frac{x^{2-\delta}}{\log^{2}x}
    dx
\nonumber \\ &\le
\frac{C_{c,\gamma}}{\log^2 k}
\frac{1}{z^2}
\int_{c_k}^{(1-c)z}
    \frac{1}{\log^{4}x}x^{1-2\delta}
dx
\le
\frac{C_{c,\gamma,M}}{\log^{2}k\cdot z^{2\delta}\log^{4}z}
.\label{eq:col_kU_chebyshev}
\end{align}
Meanwhile, for $z\le (1-\sigma)\psi(M)$,
\begin{equation}\label{eq:col_kU_part2}
    kP(U\ge z) \ge k\cdot \sigma z f_U((1+\sigma)z)
    =
    \frac{C_{\gamma,\sigma,M}}{\log k\cdot z^\delta\log^2 z}.
\end{equation}
Comparing \eqref{eq:col_kU_chebyshev} and \eqref{eq:col_kU_part2} and using Fact~\ref{fact:col_sumzm_stableintegr}(1), we have for all $z\ge Nt_k$ that
\begin{align}
\int_{0}^{(1-c)z}
f_{U_{(1)}}(x)
P\bigg(\sum_{i=1}^{n}U_{i}\ge z\mid U_{(1)}=x\bigg)dx \le
C_{c,\gamma,\sigma,M}N^{-\delta} kP(U\ge z)
.\label{eq:col_kU_sumU_part2}
\end{align}
Combine (\ref{eq:col_kU_sumU_part1}) and (\ref{eq:col_kU_sumU_part2}). For each $\epsilon>0$, we can first pick $c \le \epsilon/2C_{\sigma,M}$ and then choose $N=N(M,\gamma,\epsilon,\sigma,c)$ such that (\ref{eq:col_kU_Nconstraint1})
is true and for all $k\ge k_0$, $z\in[Nt_{k},(1-\sigma)\psi(M)]$,
\begin{align}
P\bigg(\sum_{i=1}^{k}U_{i}\ge z\bigg)
&\le kP(U\ge z)
\bigg(
    1 + C_{\sigma,M}\cdot c
    + \frac{C_{c,\gamma,\sigma,M}}{N^{\delta}}
\bigg)
\le (1+\epsilon)kP(U\ge z)
.\label{eq:col_kU_case1}
\end{align}
\item $z\in[c_k,Nt_{k}]$:
Lemma~\ref{lem:col_sumzm_stablelaw} implies that for $z'\in(1,N]$, $P(\sum_{i=0}^k U_i \ge z't_k)$ converges uniformly to $1 \wedge P(\widetilde{U}>z')$ as $k\to\infty$ and $\widetilde{U}$ follows the Levy distribution with parameter $\frac{\pi}{2}$.
Comparing the $c_{k}$ in the RHS of \eqref{eq:col_sumzm_ktimesU} to the definition of $t_k$ yields that $c_k\ge t_k$ for any $\epsilon>0$.
Therefore  for $k>k_0$ and $z\in[c_{k},Nt_{k}]$ with $z'=z/t_{k}\in(1,N]$,
\begin{equation}
P\bigg(\sum_{i=1}^{k}U_{i}\ge z\bigg)
\le(1+\epsilon/2)P\big(\tilde{U}>z'\big)
\le(1+\epsilon)kP(U\ge z't_{k})
,\label{eq:col_kU_case2}
\end{equation}
where the last step uses \eqref{eq:col_sumzm_generaldelta}.
\item Finally using \eqref{eq:col_kU_case1} and recall the definition of $\sigma$, we have for all $z\ge(1-\sigma)\psi(M)$ that
\begin{equation}\label{eq:col_kU_case3}
P\bigg(\sum_{i=1}^{k}U_{i}\ge z\bigg)
\le P\bigg(\sum_{i=1}^{k}U_{i}\ge(1-\sigma)\psi(M)\bigg)
\le (1+\epsilon)k P(U\ge (1-\sigma)\psi(M))
\le \frac{\epsilon}{\log k}.
\end{equation}
\end{enumerate}
Combining \eqref{eq:col_kU_case1}, \eqref{eq:col_kU_case2} and \eqref{eq:col_kU_case3} completes the proof.
\end{proof}

\subsection{Distribution of \texorpdfstring{$\log(\sum_{m=2}^{k}\exp(Z_{1}-Z_{m}))$}{log sum exp(Z1-Zm)}}
In this section we bound the distribution of $W_0:=-\log(\sum_{m=2}^{k}e^{Z_{1}-Z_{m}})$. First we rewrite (\ref{eq:col_zm_z1}) as
\[
Z_{1}\prec
\left(
    \pois\left(\frac{1}{2}\kappa\right)\otimes\delta_{\alpha}
\right)\oplus\left(
    \pois(\gamma p_{r}^{\neq})
    \otimes\frac{1}{\gamma p_{r}^{\neq}}\nu_{r}^{\neq}
\right)
=:R_{0}+R_{r}=:\widetilde{Z}_{1}
    ,
\]
and let $\tilde{\nu}_{-Z_{1}}$ be the distribution of $-\widetilde{Z}_{1}$. Then we define $V:=-\log(\sum_{m=2}^{k}e^{-Z_{m}})$. The conclusion of Lemma~\ref{lem:col_sumzm} can be rewritten as
\begin{align}
\nu_{V}
&\succ
\frac{e^{\gamma+1-\beta}}{\log k}
\left[\psi^{-1}\circ (\psi+\sigma)\circ\nu_{S_{1}}+(1+\epsilon)C_{Z}\nu_{r}\right]1_{\le a_{k}}
+\frac{\epsilon}{\log k}\delta_{-\infty}
\nonumber\\&
=:\tilde{\nu}_{V}^{1}+\tilde{\nu}_{V}^{r}+\tilde{\nu}_{V}^{\infty}
=:\tilde{\nu}_{V}
.\label{eq:col_W_nuV}
\end{align}
Let $\widetilde{V}$ be sampled from $\tilde{\nu}_{V}$.
  Note that $Z_1$ is independent of $\sum_{m=2}^k Z_m$. We finally define
\begin{equation}
\widetilde{W}_{0}
:= \widetilde{V}-\widetilde{Z}_{1}
\prec V-Z_{1} = W_0
,\label{eq:col_W_tildeW}
\end{equation}
\begin{lem}
\label{lem:col_W_Wbound} Assume that $(\delta,\kappa,\alpha_{0},M,\sigma,\gamma,\epsilon)$ satisfies the conditions of Lemma~\ref{lem:col_zm_S0S1bound} and \ref{lem:col_sumzm}.
\begin{enumerate}
\item If $\delta\le\frac{1}{2}$, then there exists constant $C_{\delta,M}>0$ such that for each $y\ge M$,
\[
(\nu_{r}*\tilde{\nu}_{-Z_{1}})(dy)\le(1+C_{\delta,M}\gamma)\exp(\kappa(e^{\alpha\delta}-1)/2)\nu_{r}(dy).
\]

\item There exists constant $C_{\delta,\alpha,M}^{\star}>0$ such that $(\nu_{r}*\tilde{\nu}_{-Z_{1}})((-\infty,M])\le\gamma\cdot C_{\delta,\alpha,M}^{\star}$.
\item For any fixed $\kappa,\alpha_{0}$ and $y_{1},y_{2}\ge M$,
\[
\liminf_{\sigma,\gamma\to0}
(\tilde{\nu}_{V}^{1}*\tilde{\nu}_{-Z_{1}}) ([y_{1},y_{2}])
\ge\frac{e^{-\kappa/2}}{2\log k}
P\big(\pois(\kappa)\cdot\alpha\in(y_{1},y_{2})\big)
.
\]

\end{enumerate}
\end{lem}
\begin{proof}
\textbf{Part 1: }By definition, for any $y\ge M$
\begin{equation}
\nu_{r}*\tilde{\nu}_{-Z_{1}}(dy)
= \int_{-\infty}^{0}
\frac{\gamma e^{\delta(y-z)}}{(y-z)^{2}}
\tilde{\nu}_{-Z_{1}}(dz)
\le \frac{\gamma e^{\delta y}}{y^{2}}dy
\cdot
\int_{-\infty}^{0}e^{-\delta z}\tilde{\nu}_{-Z_{1}}(dz)
= \nu_{r}(dy) \E e^{\delta\widetilde{Z}_{1}}
.
\label{eq:col_W_part1}
\end{equation}
Hence it is enough to bound $\E\exp(\delta\widetilde{Z}_{1})=\E\exp(\delta R_{0})\E\exp(\delta R_{r}).$
For the first term,
\begin{equation}
\E\exp(\delta R_{0})=\E\exp\big(\delta\alpha\cdot\pois(\kappa/2)\big)=\exp\big(\kappa(e^{\alpha\delta}-1)/2\big).\label{eq:col_W_T0}
\end{equation}
For the second term, $R_{r}$ has the same distribution as the sum of points from the Poisson point process with intensity $\nu_{r}^{\neq}(dy)$. Recall that
\[
\nu_{r}^{\neq}(dy)
=\left(e^{y}+(k-1)^{-1}\right)^{-1} \nu_{r}(dy)
\le \frac{\gamma}{y^{2}}e^{(\delta-1)y}dy
\]
and $p_{r}^{\neq}=\frac{1}{\gamma}\nu_{r}^{\neq}([M,\infty))$ depends
only on $\delta,M$. By Campbell's Theorem, for any $\delta\le\frac{1}{2}$
and $\gamma\le1$,
\begin{equation}
\E\exp(\delta R_{r})
= \exp\left( \int_{M}^{a_{k}}
    (e^{\delta y}-1)\nu_{r}^{\neq}(dz)
\right)
\le \exp\left(
    \gamma\int_{M}^{\infty}y^{-2}e^{(2\delta-1)y}dy
\right)
\le 1+\gamma C_{\delta,M},\label{eq:col_W_Tr}
\end{equation}
where in the last step we use the inequality $e^{x}\le1+xe^{x},\forall x\ge0$.
Plugging (\ref{eq:col_W_T0}) and (\ref{eq:col_W_Tr}) back into (\ref{eq:col_W_part1})
yields the desired result.

\textbf{Part 2: }Expanding the convolution of $\nu_{r}*\tilde{\nu}_{-Z_{1}}$ yields that
\begin{align*}
\nu_{r}*\tilde{\nu}_{-Z_{1}}((-\infty,M])
&\le \int_{0}^{\infty}
\int_{M}^{z+M}\frac{\gamma}{y^{2}}e^{\delta y}
\cdot \tilde{\nu}_{Z_{1}}(dz) dy
\le \frac{\gamma e^{\delta M}}{\delta M^{2}}
\int_{0}^{\infty}e^{\delta z}\tilde{\nu}_{Z_{1}}(dz)
= \frac{\gamma e^{\delta M}}{\delta M^{2}}
\E e^{\delta\widetilde{Z}_{1}}
.
\end{align*}
Applying (\ref{eq:col_W_T0}) and (\ref{eq:col_W_Tr}) to $\E e^{\delta \widetilde{Z}_1}$ gives one possible $C^\star_{\delta,\alpha,M} = \exp(e^{\alpha\delta}-1) (1+C_{\delta,M}) e^{\delta M}/(\delta M^2)$.

\textbf{Part 3: }Noting that $\psi^{-1}(\psi(y)+\sigma)=-\log(e^{-y}-\sigma)$,
we have that
\begin{align*}
\tilde{\nu}_{V}^{1}*\tilde{\nu}_{-Z_{1}}([y_{1},y_{2}])
&\ge\frac{e^{\gamma+1-\beta}}{\log k}
P(\widetilde{Z}_{1}=0)
\cdot P\big(\pois(\kappa)\cdot\alpha
    \in[\log(e^{-y_{1}}-\sigma),\log(e^{-y_{2}}-\sigma)]
\big)
\\& \ge\frac{1}{\log k}
e^{-\frac{1}{2}\kappa-\gamma p_{r}^{\neq}}
P\left( \pois(\kappa)\cdot\alpha
\in(\log(-e^{-y_{1}}-\sigma),-\log(e^{-y_{2}}-\sigma))
\right).
\end{align*}
$\pois(\kappa)\cdot\alpha$ takes values from the discrete set $\alpha\mathbb{Z}_+$. For any fixed $y_1,y_2$, there exists $\sigma=\sigma(\alpha,y_1,y_2)$ such that there is no points of $\alpha\mathbb{Z}_+$ between $-\log(e^{-y_i}-\sigma)$ and $y_i$, $i=1,2$. Hence in the last line we can substitute the probability by $P(\pois(\kappa)\cdot\alpha\in(y_{1},y_{2}))$.
Letting $\gamma\to0$ finishes the proof.
\end{proof}

\subsection{Final step}

Finally we are ready to prove Theorem~\ref{thm:col_rec_stocdom}.
\begin{proof}[Proof of Theorem~\ref{thm:col_rec_stocdom}]
By Proposition~\ref{prop:col_rec_rewrite}, it suffices to show that under certain choice of parameters $(\delta,\kappa,\alpha_{0},M,\sigma,\gamma,\epsilon)$,
the random variable $W$ defined in \eqref{eq:col_rec_W} stochastically dominates $\nu_{k}$ by $c/\log k$ for some fixed
$c>0$. For any $\alpha_{0}>0$ and $\alpha=\phi(\frac{1}{2}-\alpha_{0})$,
we first choose $\sigma<\sigma_{1}(\alpha_{0})$ such that $\log(1+e^{-\sigma})>\frac{1}{2}(1-\alpha_{0})$.
Thus for $k\ge k_0$ we can write
\begin{equation}
W\succ\log\left(\frac{k-2}{k-1}+\exp(\widetilde{W}_{0})\right)\vee0\ge\begin{cases}
\widetilde{W}_{0} & \widetilde{W}_{0}\ge M\\
\alpha & M>\widetilde{W}_{0}\ge-\sigma\\
0 & -\sigma>\widetilde{W}_{0}
\end{cases}
.\label{eq:col_final_W0toW}
\end{equation}
Comparing the RHS of last equation with the definition of $\nu_k$, it is suffices show that
\begin{align}
&P(\widetilde{W}_{0}<-\sigma)
\le\frac{1}{\log k}(1-\kappa) -\frac{c}{\log k}
\quad \textup{and }
\label{eq:col_final_stocpart1}
\\
&P(\widetilde{W}_{0}\le x)\le\nu_{k}([0,x])-\frac{c}{\log k}
\quad \textup{for all } x\ge 0 \textup{ such that } \nu_{k}([0,x])<1
.
\label{eq:col_final_stocpart2}
\end{align}
Recall the three parts of $\tilde{\nu}_V$ in \eqref{eq:col_W_nuV} and define $\tilde{\nu}_{W_{0}}^{\bullet}(dx):=\tilde{\nu}_{V}^{\bullet}*\tilde{\nu}_{-Z_{1}}(dx)$ for $\bullet \in \set{1,r,\infty}$.
Figure~\ref{fig:pic} gives an illustration of $\tilde{\nu}_{W_0}$ and $\nu_k$, where bars represent the discrete parts, curves represent the continuous parts and the left two dotted boxes corresponds to last two cases of \eqref{eq:col_final_W0toW}.
\begin{figure}[t]
\begin{center}
    \includegraphics[width = \textwidth]{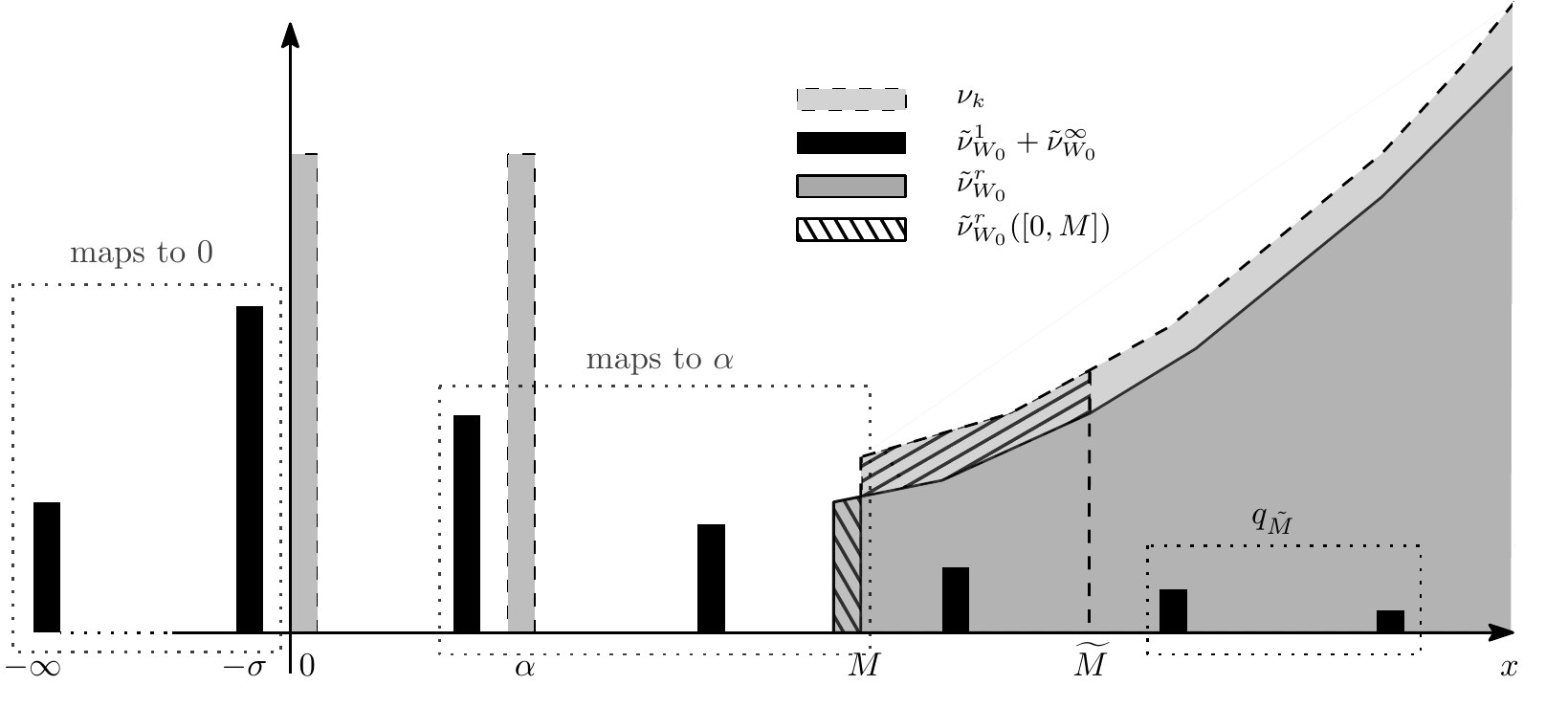}
    \caption{$\tilde{\nu}_{W_0}$ and $\nu_k$}\label{fig:pic}
\end{center}
\end{figure}
Fix $\delta=\frac{1}{2}$. To show \eqref{eq:col_final_stocpart1} is to show that the weight in the first dotted box is strictly smaller than $\nu_k(\set{0})=\kappa$. We set $\kappa=\frac{1}{2}$ such that
\[
P(\pois(\kappa/2)=0)=e^{-1/4}>\frac{3}{4}>\frac{1}{2}=\kappa.
\]
Recall the definition of $C_Z=C_Z(\delta,\kappa,\alpha_0,M,\gamma)$ in \eqref{eq:col_sumzm_CZ}. By Lemma~\ref{lem:col_W_Wbound}(2), for each fixed $\delta,\kappa,\alpha_0,M$, we can choose  $\epsilon_0,\gamma_0,\beta_0$ such that for all $\epsilon<\epsilon_{0},\gamma<\gamma_{0}$, $\beta_0<\beta<1$ and $c_{0}=\frac{1}{10}$,
\begin{align*}
P(\widetilde{W}_{0}<-\sigma)
&\le \frac{e^{\gamma+1-\beta}}{\log k}
\left[
    P(\widetilde{Z}_{1}\neq 0)
    +(1+\epsilon)C_{Z}\gamma\cdot C_{\delta,\alpha,M}^{*}
    +\epsilon
\right]
\\&\le
\frac{e^{\gamma+1-\beta}}{\log k}
\left[
    1-e^{-\frac{1}{2}\kappa-\gamma p_{\ne}^{r}}
    +2C_Z C_{\delta,\alpha,M}^{*}\gamma
    +\epsilon
\right]
\\&=
\frac{4/3}{\log k}
\left[
    \frac{1}{4}+\epsilon+o_{\gamma}(1)
\right]
<\frac{2}{5}\frac{1}{\log k}
=\left(\frac{1}{2}-c_{0}\right)\frac{1}{\log k}.
\end{align*}

The proof of \eqref{eq:col_final_stocpart2} is roughly done in three parts. We first show that the asymptotically, $\tilde{\nu}_{W_0}^r$ is smaller than $\nu_k^r$ by a multiplicative constant factor.
Then we show that the underflow of $\tilde{\nu}_{W_0}^r$ below $M$ (the vertical stripped area in Figure~\ref{fig:pic}) can be compensated by the overflow of $\tilde{\nu}_{W_0}^1$ above $M$ (the $q_{\widetilde{M}}$ box in Figure~\ref{fig:pic}). Finally we make sure that the compensation is can be absorbed into the gap of $\tilde{\nu}_{W_0}^r$ and $\nu_k^r$ (the wide stripped area in Figure~\ref{fig:pic}).

We first look at sufficiently large values of $x$.
By Lemma~\ref{lem:col_W_Wbound}(1),
\begin{equation}
\tilde{\nu}_{W_0}^r(dx)
\le\frac{e^{\gamma+1-\beta}}{\log k}
(1+\alpha_{0})(1+C_{\delta,M}(\gamma+\epsilon))
\exp(\kappa(e^{\alpha\delta}+2e^{-\alpha\delta}-3)/2)\nu_{r}(dx)
,\quad \forall x\ge M
.\label{eq:col_final_tail}
\end{equation}
Let $\alpha_{0}$ be a small constant such that (note that $\phi(\frac{1}{2})=\log2-o_{k}(1)$
and $\exp(\sqrt{2}-3/2)\approx0.92<\frac{12}{13}$)
\[
(1+\alpha_{0})\exp(\kappa(e^{\alpha\delta}+2e^{-\alpha\delta}-3)/2)=(1+o_{\alpha_{0}}(1))\exp(\sqrt{2}-3/2)<\frac{12}{13}<1,
\]
and let $M>M(\alpha_{0})\vee\frac{2}{\delta}$ such that Lemma~\ref{lem:col_zm_S0S1bound}
is satisfied. Recall the definition of constant $C_{\delta,M}$ from the constants in Lemma~\ref{lem:col_zm_S0bound} and Lemma~\ref{lem:col_W_Wbound}.
Given our choice of $\delta,\kappa,\alpha_0,M$ so far,
we can choose $\gamma_{1},\epsilon_1,\beta_1$ such that for all $\gamma\le\gamma_{1},\epsilon\le\epsilon_{1}$, $1-\beta<1-\beta_{1}$ and all $x\ge M$,
\begin{equation}
\textup{RHS of (\ref{eq:col_final_tail})}
\le\frac{12}{13}\cdot\frac{1}{\log k}e^{1+\gamma-\beta}\nu_{r}(dx)\le\frac{14}{15}\frac{1}{\log k}\nu_{r}(dx).\label{eq:col_final_taildensity2}
\end{equation}

Next we consider the values of $x$ near $M$. We first choose $\widetilde{M}=\widetilde{M}(\delta,\alpha,M)>M\vee2\alpha$ such that
\begin{equation}
\frac{1}{15}\nu_{r}([M,\widetilde{M}])=\frac{1}{15}\int_{M}^{\widetilde{M}}\frac{\gamma}{y^{2}}e^{\delta y}dy\ge e^{\gamma_{1}+1-\beta_{1}}(C_{\delta,\alpha,M}^{*}+2e^{\gamma_{1}}+\epsilon_{1})\gamma
,\label{eq:col_final_tildeMproperty}
\end{equation}
where $C_{\delta,\alpha,M}^{*}$ is the constant in Lemma~\ref{lem:col_W_Wbound}(2).
Let $q_{\widetilde{M}}:=\frac{1}{2}P(\pois(\kappa)\cdot\alpha\in(\widetilde{M},2\widetilde{M}))$. $q_{\widetilde{M}}$ is strictly positive since $\widetilde{M}>2\alpha$. By Lemma~\ref{lem:col_W_Wbound}(3), we can choose $\sigma_2,\gamma_2$ such that for all $\sigma<\sigma_{2}$, $\gamma<\gamma_2$,
\begin{equation}
\tilde{\nu}_{W_{0}}^{1}([\widetilde{M},2\widetilde{M}])
=\tilde{\nu}_{V}^{1}*\tilde{\nu}_{-Z_{1}}
    ([\widetilde{M},2\widetilde{M}])
\ge q_{\widetilde{M}}>0
.\label{eq:col_final_qMlowerbound}
\end{equation}
We further choose $\gamma_3,\epsilon_3,\beta_2$ such that for all $\gamma\le\gamma_{3}$, $\epsilon\le\epsilon_{2}<1$, $1-\beta\le1-\beta_{2}$ and some $c_{1}\in(0,q_{\tilde{M}})$,
\begin{equation}
e^{\gamma+1-\beta}
\left[
    (1-q_{\tilde{M}})+\gamma C_{\delta,\alpha,M}^{*}+\epsilon
\right]
\le 1-c_{1} <1
.\label{eq:col_final_qMconstraint}
\end{equation}
(\ref{eq:col_final_taildensity2}), (\ref{eq:col_final_qMlowerbound})
and (\ref{eq:col_final_qMconstraint}) together implies for $x\le\widetilde{M}$,
(note that $\nu_{k}([0,M])=1/\log k$)
\begin{align*}
 & \tilde{\nu}_{W_{0}}([-\infty,x]):=(\tilde{\nu}_{W_{0}}^{1}+\tilde{\nu}_{W_{0}}^{r}+\tilde{\nu}_{W_{0}}^{\infty})([-\infty,x])\\
 & \le\frac{e^{\gamma+1-\beta}}{\log k}\left((1-q_{\widetilde{M}})+\gamma C_{\delta,\alpha,M}^{*}+\epsilon\right)+\frac{14}{15}\frac{1}{\log k}\nu_{r}([M,x\vee M])\\
 & \le\frac{1-c_{1}}{\log k}+\frac{14}{15}\frac{1}{\log k}\nu_{r}([M,x\vee M])\le\nu_{k}([0,x])-\frac{c_{1}}{\log k}.
\end{align*}

Finally, for $x\ge\widetilde{M}$ such that $\nu_{k}([0,x])<1$, we can choose $c_{2},\beta_3$ such that for $\gamma=(\gamma_{0}\wedge\gamma_{1}\wedge\gamma_{2}\wedge\gamma_{3})$ and $1-\beta<1-\beta_{3}$, we have $e^{\gamma+1-\beta}+c_{2}<1+2\gamma e^{\gamma}$.
Using (\ref{eq:col_final_tildeMproperty}), we have
\begin{align*}
\tilde{\nu}_{W_{0}}([-\infty,x]) & \le\frac{e^{\gamma+1-\beta}}{\log k}\left(1+\gamma C_{\delta,\alpha,M}^{*}+\epsilon\right)+\frac{14}{15}\frac{1}{\log k}\nu_{r}([M,\widetilde{M}])+\frac{14}{15}\frac{1}{\log k}\nu_{r}((\widetilde{M},x])\\
 & \le\frac{1}{\log k}+\frac{1}{\log k}\big(e^{\gamma+1-\beta}(1+\gamma C_{\delta,\alpha,M}^{*}+\epsilon)-1-\frac{1}{15}\nu_{r}([M,\widetilde{M}])\big)+\frac{1}{\log k}\nu_{r}([M,x])\\
 & \le\frac{1-c_{2}}{\log k}+\frac{1}{\log k}\nu_{r}([M,x])=\nu_{k}([0,x])-\frac{c_{2}}{k\log k}.
\end{align*}

Combining all pieces together, we have the desired result with $\delta,\kappa,\alpha_0,M,\gamma$ set as specified before,
$\sigma=\sigma_{1}\wedge\sigma_{2}$,
$\epsilon=\epsilon_{0}\wedge\epsilon_{1}\wedge\epsilon_{2}$,
and
$\beta^{0}=\beta_{0}\vee\beta_{1}\vee\beta_{2}\vee\beta_{3}$,
$c=c_{0}\wedge c_{1}\wedge c_{2}$
.
\end{proof}

\section{Appendix\label{sec:appd}}
\begin{proof}[Proof of Fact \ref{fact:col_zm_int_x}]
First fix $n=2$ and $t'\ge2M$. For each $x_{1}\ge M$, either $x_{1}$
or $t'-x_{1}$ is larger than $t'/2$, hence
\begin{equation}
\int_{M}^{t'-M}\frac{1}{x_{1}^{2}(t-x_{1})^{2}}dx_{1}\le\frac{2}{(t'/2)^{2}}\int_{M}^{\infty}\frac{1}{x_{1}^{2}}dx_{1}=\frac{8}{Mt'^{2}}.\label{eq:z0_fact}
\end{equation}
Recursively apply (\ref{eq:z0_fact}) with $t'=t-\sum_{i=1}^{n-j}x_{i},j=2,\dots,n-1$,
we have
\begin{align*}
 & \int_{x_{i}\ge M,\sum_{i=1}^{n-1}x_{i}\le t-M}\frac{1}{x_{1}^{2}\cdots x_{n-1}^{2}(t-\sum_{i=1}^{n-1}x_{i})^{2}}dx_{1}\cdots dx_{n-1}\\
 & =\int_{x_{i}\ge M,\sum_{i=1}^{n-2}x_{i}\le t-2M}\frac{1}{x_{1}^{2}\cdots x_{n-2}^{2}}\left(\int_{M}^{t-\sum_{i=1}^{n-2}x_{i}-M}\frac{1}{x_{n-1}^{2}(t-\sum_{i=1}^{n-1}x_{i})^{2}}dx_{n-1}\right)dx_{1}\cdots dx_{n-2}\\
 & \le\frac{8}{M}\int_{x_{i}\ge M,\sum_{i=1}^{n-2}x_{i}\le t-M}\frac{1}{x_{1}^{2}\cdots x_{n-2}^{2}(t-\sum_{i=1}^{n-2}x_{i})^{2}}dx_{1}\cdots dx_{n-2}\le\cdots\le(\frac{8}{M})^{n}\frac{1}{t^{2}}.
\end{align*}

\end{proof}

\begin{proof}[Proof of Fact \ref{fact:col_sumzm_stableintegr}]
Let $s_{k}=(\frac{\gamma\delta}{\log k(\log\log k)^{2}})^{1/\delta}$,
it is easy to check that
\[
\frac{\gamma}{\delta}s_{k}^{-\delta}\log^{-2}s_{k}=(1+o_{k}(1))\log k.
\]
 For any $\epsilon>0$, let $c$ be large enough such that $(1-\epsilon)^{\delta}-2c^{-\delta}>1$. It follows that
\begin{align*}
\int_{(1-\epsilon)s_{k}}^{\infty}k\log k\cdot\mu_{U}(dx) & =\int_{(1-\epsilon)s_{k}}^{\psi(M)}\frac{\gamma}{(\log x)^{2}}x^{-(1+\delta)}dx\ge\frac{\gamma}{\log^{2}(1-\epsilon)s_{k}}\int_{(1-\epsilon)s_{k}}^{cs_{k}}x^{-(1+\delta)}dx\\
 & =\frac{\gamma}{\delta\log^{2}(1-\epsilon)s_{k}}s_{k}^{-\delta}((1-\epsilon)^{-\delta}-c^{-\delta})>(1+c^{-\delta}+o_{k}(1))\log k.
\end{align*}
Therefore $t_{k}>(1-\epsilon)s_{k}$ for $k\ge k_0$. In the other direction, let $s'_{k}=(c'\log k)^{-1/\delta}$ for some large constant $c'>0$, $\log(s_{k}')=(1+o_{k}(1))\frac{1}{\delta}\log\log k=(1+o_{k}(1))\log s_{k},$ we have
\begin{align*}
\int_{(1+\epsilon)s_{k}}^{\infty}k\log k\cdot\mu_{U}(dx) & =\int_{(1+\epsilon)s_{k}}^{\psi(M)}\frac{\gamma}{(\log x)^{2}}\frac{1}{x^{1+\delta}}dx\\
 & \le\frac{\gamma}{\log^{2}\psi(M)}\int_{s'_{k}}^{\infty}x^{-(1+\delta)}dx+\frac{\gamma}{\log^{2}(s'_{k})}\int_{(1+\epsilon)s_{k}}^{\infty}x^{-(1+\delta)}dx\\
 & \le\frac{\gamma}{\delta\log^{2}\psi(M)}c'^{-\delta}\log k+(1+o_{k}(1))(1+\epsilon)^{-\delta}\log k.
\end{align*}
Let $c'$ be large enough such that $\frac{\gamma}{\delta\log^{2}\psi(M)}c'^{-\delta}+(1+\epsilon)^{-\delta}<1-c'^{-1}<1$,
we have for $k\ge k_0$ that $t_{k}<(1+\epsilon)s_{k}$. This completes the Part 1. Part 2 can be derived similarly.
\end{proof}
\bibliographystyle{plain}
\bibliography{ref}

\end{document}